\documentclass[11pt, reqno, oneside]{amsart}
\usepackage{amsmath,amsfonts,amssymb,amsthm}

\usepackage{enumerate}
\usepackage{esint}
\usepackage[pagebackref,colorlinks,citecolor=blue,linkcolor=blue]{hyperref}
\usepackage[dvipsnames]{xcolor}
\usepackage{mathtools}  

\usepackage{mathtools}
\usepackage{setspace}

\usepackage{comment}

\usepackage{geometry}
\makeatletter
\@namedef{subjclassname@2020}{%
  \textup{2020} Mathematics Subject Classification}
\makeatother

\numberwithin{equation}{section}

\theoremstyle{plain}
\newtheorem{theorem}[equation]{Theorem}
\newtheorem{corollary}[equation]{Corollary}
\newtheorem{lemma}[equation]{Lemma}
\newtheorem{proposition}[equation]{Proposition}

\theoremstyle{definition}
\newtheorem{definition}[equation]{Definition}

\newtheorem{example}[equation]{Example}
\newtheorem{remark}[equation]{Remark}

\newcommand{\R}{{\mathbb R}}
\newcommand{\N}{{\mathbb N}}

\newcommand{\Om}{\Omega}

\providecommand{\vint}[1]{\mathchoice
          {\mathop{\vrule width 5pt height 3 pt depth -2.5pt
                  \kern -9pt \kern 1pt\intop}\nolimits_{\kern -5pt{#1}}}
          {\mathop{\vrule width 5pt height 3 pt depth -2.6pt
                  \kern -6pt \intop}\nolimits_{\kern -3pt{#1}}}
          {\mathop{\vrule width 5pt height 3 pt depth -2.6pt
                  \kern -6pt \intop}\nolimits_{\kern -3pt{#1}}}
          {\mathop{\vrule width 5pt height 3 pt depth -2.6pt
                  \kern -6pt \intop}\nolimits_{\kern -3pt{#1}}}}

\newcommand{\eps}{\varepsilon}
\newcommand{\loc}{\mathrm{loc}}

\newcommand{\BV}{\mathrm{BV}}

\newcommand{\ch}{\text{\raise 1.3pt \hbox{$\chi$}\kern-0.2pt}}

\DeclareMathOperator{\Mod}{Mod}
\DeclareMathOperator{\AM}{AM}

\DeclareMathOperator{\diam}{diam}

\begin{document}
\title{Quasiconformal and Sobolev mappings\\
in non-Ahlfors regular metric spaces
}
\author{Panu Lahti}
\address{Panu Lahti,  Academy of Mathematics and Systems Science, Chinese Academy of Sciences,
	Beijing 100190, PR China, {\tt panulahti@amss.ac.cn}}
\author{Xiaodan Zhou}
\address{Xiaodan Zhou, Analysis on Metric Spaces Unit, Okinawa Institute of Science and Technology Graduate University, 1919-1, Onna-son, 
Okinawa 904-0495, Japan, {\tt xiaodan.zhou@oist.jp}}


\begin{abstract}
We show that a mapping $f\colon X\to Y$
satisfying the metric condition of quasiconformality outside suitable exceptional sets
is in the Newton-Sobolev class $N_{\loc}^{1,1}(X;Y)$. Contrary to previous works, we only assume
an asymptotic version of Ahlfors-regularity on $X,Y$. This allows many non-Ahlfors regular spaces,
such as weighted spaces and Fred Gehring's bowtie, to be included in the theory.
Unexpectedly, already in the classical setting of unweighted Euclidean spaces,
our theory detects Sobolev mappings that are not recognized by previous results.
\end{abstract}

\date{\today}
\maketitle

\section{Introduction}

Given two metric spaces $(X,d)$ and $(Y,d_Y)$ and a mapping $f\colon X\to Y$,
for every $x\in X$ and $r>0$ one defines
\[
L_f(x,r):=\sup\{d_Y(f(y),f(x))\colon d(y,x)\le r\}
\]
and
\[
l_f(x,r):=\inf\{d_Y(f(y),f(x))\colon d(y,x)\ge r\},
\]
and then
\[
H_f(x,r):=\frac{L_f(x,r)}{l_f(x,r)};
\]
we interpret this to be $\infty$ if the denominator is zero.
A homeomorphism $f\colon X\to Y$ is (metric) \emph{quasiconformal} if there is a number
$1\le H<\infty$ such that
\begin{equation}\label{eq:Hf}
H_f(x):=\limsup_{r\to 0} H_f(x,r)\le H
\end{equation}
for all $x\in X$.

Gehring \cite{Ge62,Ge} proved that a quasiconformal mapping
$f\colon \R^n\to \R^n$
is absolutely continuous on almost every line.
Generalizations of this result in Euclidean and Carnot-Carath\'eodory spaces include the work
by Marguilis--Mostow \cite{MaMo},
Balogh--Koskela \cite{BaKo},
Kallunki--Koskela \cite{KaKo},
Kallunki--Martio \cite{KaMa}, and
Koskela--Rogovin \cite{KoRo}.
In these papers it is shown that one does not need the condition $H_f(x)\le H$ at \emph{every}
point $x$, and that instead of $H_f$ one can consider
\[
h_f(x):=\liminf_{r\to 0} H_f(x,r).
\]

Additionally, one can consider more general metric measure spaces $X,Y$;
as noted by Balogh--Koskela--Rogovin \cite{BKR},
it is then a problem of general interest to find minimal assumptions
on the spaces $X,Y$, and on the mapping $f$, which guarantee absolute continuity on almost all curves.
These authors proved in \cite[Theorem 1.1]{BKR} that a homeomorphism $f$
belongs to the Newton--Sobolev space $N_{\loc}^{1,1}(X;Y)$ provided that
\begin{enumerate}
\item the metric measure spaces $X,Y$ are locally Ahlfors $Q$-regular, $Q>1$;
\item there exists an exceptional set $E$ of $\sigma$-finite $\mathcal H^{Q-1}$-measure such
	that $h_f<\infty$ on $X\setminus E$; 
\item $h_f\le H<\infty$ almost everywhere.
\end{enumerate}
In particular, unlike in previous works in metric spaces, by Heinonen--Koskela \cite{HK1,HK2}
and Heinonen--Koskela--Shanmugalingam--Tyson \cite[Theorem 9.8]{HKST},
no Poincar\'e inequality is assumed.
Condition (3) amounts to saying that $h_f\in L^{\infty}(X)$.
Williams \cite{Wi} improved this result by showing that it is in fact sufficient to replace condition (3) with
$h_f\in L_{\loc}^{Q/(Q-1)}(X)$.
However, in all of these papers, (local) Ahlfors $Q$-regularity is assumed.

We show that it is possible to further weaken all of these assumptions,
in particular largely remove the assumption of Ahlfors regularity.
Our main result is the following.

\begin{theorem}\label{thm:main theorem intro}
Let $(X,d,\mu)$ and $(Y,d_Y, \nu)$ be complete metric measure spaces. 
	Let $\Om\subset X$ be open and bounded.
	Assume $f\colon \Om\to Y$ is injective and continuous with $\nu(f(\Om))<\infty$, and:
	\begin{itemize}
	\item[(1)] Suppose that $(\Om,d)$ is metric doubling,
	and that there exists a set $E\subset\Om$ such
	that $\mu$ is asymptotically doubling in
	$\Om\setminus E$,
	and that in
	$\Om\setminus E$ there exist Borel functions $Q(x)>1$ and $R(x)>0$ such that
	\begin{equation}\label{eq:compatibility in intro}
	\limsup_{r\to 0}\frac{\mu(B(x,r))}{r^{Q(x)}}<R(x)\liminf_{r\to 0}\frac{\nu(B(f(x),r))}{r^{Q(x)}}
	\quad \textrm{for every }x\in \Om\setminus E.
	\end{equation}
		
	\item[(2)] Suppose that $E$ is the union of a countable set and a set
	with $\sigma$-finite codimension $1$ Hausdorff measure $\widehat{\mathcal H}$.
	
	\item[(3)] Finally assume that
	\[
	\frac{Q(\cdot)-1}{Q(\cdot)}(R(\cdot)\max\{h_f(\cdot),1\}^{Q(\cdot)})^{1/(Q(\cdot)-1)}\in L^1(\Om).
	\]
	\end{itemize}
	Then $f\in D^{1}(\Om;Y)$.
\end{theorem}

Here $D^1(\Om;Y)$ is the Dirichlet space, that is, $f$ is not required to be in $L^1(\Om;Y)$.
Note that in connected spaces always $h_f\ge 1$, but in our generality we cannot rule out
the rather anomalous case $h_f<1$, so we take the maximum.

The three assumptions in Theorem \ref{thm:main theorem} are much weaker than the
corresponding three assumptions in \cite[Theorem 1.1]{BKR}.
We compare them in the following remarks.

\begin{remark}
To the best of our knowledge, all previous results obtaining the absolute continuity on curves
for quasiconformal mappings assume 
	 Ahlfors regularity of the space. Our main theorem
	 replaces the Ahlfors regularity by a much relaxed condition. 
	The ``dimension'' $Q(x)$ is allowed to vary from
	point to point, but there needs to be a compatibility between
	the dimension of $X$ at a point $x$ and the dimension of $Y$ at $f(x)$, as given by
	\eqref{eq:compatibility in intro}.
	This allows including in the theory various spaces that are often considered in analysis on metric spaces,
	but have largely been excluded from the theory of quasiconformal mappings.
	Examples include weighted spaces as well as spaces achieved by
	``glueing'' spaces of possibly different dimensions,
	such as Fred Gehring's bowtie, see Examples \ref{ex:weights} and \ref{ex:bowtie}.
\end{remark}

\begin{remark}
Since we work in spaces without a specific dimension,
we consider a codimension $1$ Hausdorff measure $\widehat{\mathcal H}$, which however
essentially reduces to the $Q-1$-dimensional
Hausdorff measure in the Ahlfors $Q$-regular case.
The set $E$ can also contain any countable set; in our generality
even a single point could have infinite codimension $1$ Hausdorff measure.
\end{remark}

\begin{remark}
As mentioned before, Williams \cite{Wi} already showed that it is enough to
assume $h_f\in L_{\loc}^{Q/(Q-1)}(X)$, and our Condition (3) essentially reduces to this in the Ahlfors
$Q$-regular case, $Q>1$. But Condition (3) also takes into account the variation in the dimension $Q$
and in the ``density'' $R$, which may occur in our generality.
\end{remark}

Balogh--Koskela--Rogovin \cite{BKR} and Williams \cite{Wi}
both rely on a well-established method of constructing a sequence
of ``almost upper gradients''
$\{\rho_i\}_{i=1}^{\infty}$ that is bounded in $L^p(K)$ for compact sets $K\subset X$, for some $1< p\le Q$.
As noted by Williams, Ahlfors $Q$-regularity with $Q>1$ is unavoidable for this to work. 
Then by the reflexivity of $L^p$-spaces, $p>1$, a subsequence converges to a $p$-weak upper gradient of $f$.
However, in order to establish that $f\in N_{\loc}^{1,1}(X;Y)$,
one of course only needs an upper gradient belonging to $L_{\loc}^1(X)$, and not $L_{\loc}^p(X)$.
This raises the question of whether one really needs to assume Ahlfors $Q$-regularity with
$Q>1$, but on the other hand there are counterexamples when $Q=1$, see Beurling--Ahlfors \cite{BA}.

To circumvent the problem arising in the case $Q=1$, we essentially require $Q(x)>1$
\emph{pointwise}.
We also construct a sequence of ``almost upper gradients'' $\{\rho_i\}_{i=1}^{\infty}$,
but since we do not assume Ahlfors regularity, at first we only
show the sequence to be bounded in $L^1(\Om)$.
Then, due to the lack of reflexivity of $L^1$, we need a careful analysis to show that the sequence is
\emph{equi-integrable}. This guarantees the existence of a weakly converging subsequence,
so that at the limit we obtain a $1$-weak upper gradient for $f$, establishing $f\in D^1(\Om;Y)$.

In Theorem \ref{thm:main theorem} we give the main theorem in an even more general form.
In particular, our formulation makes it easy to consider spaces obtained by adding very general
weights to Ahlfors $Q$-regular spaces. As a corollary of our main theorem, we obtain the following.

\begin{corollary}\label{cor:weights}
	Let $(X_0,d,\mu_0)$ and $(Y_0,d_Y,\nu_0)$ be Ahlfors $Q$-regular spaces, with $Q>1$.
	Let $X$ and $Y$ be the same metric spaces but equipped with the weighted measures $d\mu=w\,d\mu_0$ and $d\nu=w_Y\,d\nu_0$,
	where the weights $w,w_Y$ are represented by \eqref{eq:w representative}
	and \eqref{eq:wY representative}, respectively.
	Let $\Om\subset X$ be open and bounded and let
	$f\colon \Om\to Y$ be injective and continuous, with $\nu(f(\Om))<\infty$.
	Suppose there is a set $E\subset \Om$ that is the union of a countable set and a set
	with $\sigma$-finite $\widehat{\mathcal H}$-measure,
	such that $h_f<\infty$ in $X\setminus E$.
	Suppose also that $w\colon \Om\setminus E\to (0,\infty)$ is locally bounded away from zero and
	\[
	\left(\frac{[w(\cdot)\max\{h_f(\cdot),1\}]^{Q}}{w_Y(f(\cdot))}\right)^{1/(Q-1)}\in L^1(\Om,\mu_0).
	\]
	Then $f\in D^{1}(\Om;Y)$.
\end{corollary}

We observe that very general weights $w$ are allowed in the space $X$,
and in particular we can extend \cite[Theorem 1.1]{BKR} to a wide range of weighted spaces,
see Corollary \ref{cor:Linfinity}.
Meanwhile, considering the space $Y$, we observe that $f$ belonging to the class $D^{1}(\Om;Y)$
or $N^{1,1}(\Om;Y)$
does not depend at all on what measure the space $Y$ is equipped with.
Thus we are allowed to choose \emph{any} weight $w_Y$, as long as it gives a locally finite
measure. Already on the unweighted plane, there are simple examples of $N_{\loc}^{1,1}$-mappings
$f\colon \R^2\to\R^2$ for which $h_f$ is not in $L_{\loc}^2(\R^2)$ or even in
$L_{\loc}^1(\R^2)$, and so the existing results,
e.g. those of \cite{BKR}, \cite{KoRo}, and \cite{Wi}, do not tell us that $f\in N^{1,1}_{\loc}(\R^2;\R^2)$.
However, we can detect the Sobolev property of many such mappings
from Corollary \ref{cor:weights}, simply by equipping
$Y$ with a suitable weight $w_Y$;
see Example \ref{ex:plane}.

The paper is organized as follows: we introduce notation and definitions in Section \ref{sec:Prelis}
and prove a Besicovitch-type covering lemma in Section \ref{sec:covering}.
The proof of the main result Theorem \ref{thm:main theorem}
is presented in Section \ref{sec:main theorem}, and a
discussion on the size of the exceptional set
is given in Section \ref{sec:exceptional set}.
Finally in Section \ref{sec:examples},
we give various examples of spaces that are not locally or
globally doubling or Ahlfors regular but satisfy the asymptotic condition in
Theorem \ref{thm:main theorem}, and we prove
Corollary \ref{cor:weights}.

\section{Notation and definitions}\label{sec:Prelis}

Throughout the paper, we consider two metric measure spaces
$(X,d,\mu)$ and $(Y,d_Y,\nu)$, where $\mu$ and $\nu$ are Borel regular outer measures.
We understand balls $B(x,r)$ to be open and we assume that the measure of every ball is finite
(in both spaces).
For a ball $B=B(x,r)$, we sometimes denote $2B:=B(x,2r)$;
note that in a metric space, a ball (as a set) might not have a unique center and radius, but
when using this abbreviation we will work with balls for which these have been specified.

We say that $X$ is metric doubling if there is a constant $M\in\N$ such that
every ball $B(x,r)$ can be covered by $M$ balls of radius $r/2$.
This definition works also for subsets $A\subset X$, by considering $(A,d)$ as a metric space.

We say that $\mu$ is locally doubling if for every $z\in X$
there is a scale $R>0$
and a constant $C_z\ge 1$ such that
\[
\mu(B(x,2r))\le C_z\mu(B(x,r))
\]
for all $x\in B(z,R)$ and $0<r\le R$.

Given $Q\ge 1$, we say that $\mu$ is locally Ahlfors $Q$-regular 
if for every $z\in X$ there is a scale $R>0$
and a constant $C_z\ge 1$ such that
\[
C_z^{-1}r^Q\le \mu(B(x,r))\le C_z r^Q
\]
for all $x\in B(z,R)$ and $0<r\le R$.

If these conditions hold for every $x\in X$ and $0<r<\diam X$ with $C_z$ replaced by a uniform constant,
we say that $\mu$ is doubling, or that $X$ is Ahlfors $Q$-regular.

We will only work with the following much weaker notion of a doubling measure.

\begin{definition}
We say that $\mu$ is asymptotically doubling in a set $A\subset X$ if there exists a constant
$C_A> 1$
such that
\[
\limsup_{r\to 0}\frac{\mu(B(x,2r))}{\mu(B(x,r))}< C_A\quad\textrm{for all }x\in A.
\]
\end{definition}

The $n$-dimensional Lebesgue measure is denoted by $\mathcal L^n$.
The
$s$-dimensional Hausdorff content is denoted by $\mathcal H^s_R$, $R>0$,
and the corresponding Hausdorff measure by $\mathcal H^s$;
these definitions extend automatically from Euclidean spaces to metric spaces.

A continuous mapping from a compact interval into $X$ is said to be a rectifiable curve if it has finite length.
A rectifiable curve $\gamma$ always admits an arc-length parametrization,
so that we get a curve $\gamma\colon [0,\ell_{\gamma}]\to X$
(for a proof, see e.g. \cite[Theorem 3.2]{Haj}).
We will only consider curves that are rectifiable and
arc-length parametrized.

If $\gamma\colon [0,\ell_{\gamma}]\to X$ is a curve and
$g\colon X\to [0, \infty]$ is a Borel function, we define
\[
\int_{\gamma} g\,ds:=\int_0^{\ell_{\gamma}}g(\gamma(s))\,ds.
\]
The $1$-modulus of a family of curves $\Gamma$ is defined by
\[
\Mod_{1}(\Gamma):=\inf\int_{X}\rho\, d\mu,
\]
where the infimum is taken over all nonnegative Borel functions $\rho \colon X\to [0, \infty]$
such that $\int_{\gamma}\rho\,ds\ge 1$ for every curve $\gamma\in\Gamma$.
If a property holds apart from a curve family with zero $1$-modulus, we say that it holds for
$1$-a.e. curve.
The $1$-modulus is easily seen to be subadditive.

\begin{definition}\label{def:equiintegrability}
	Given a $\mu$-measurable set $H\subset X$, a sequence of nonnegative
	functions $\{g_i\}_{i=1}^{\infty}$
	in $L^1(H)$ is equi-integrable if the following two conditions hold:
	\begin{itemize}
		\item for every $\eps>0$ there exists a $\mu$-measurable set $D\subset H$ such that $\mu(D)<\infty$ and
		\[
		\int_{H\setminus D}g_i\,d\mu<\eps \quad\textrm{for all }i\in\N;
		\]
		\item for every $\eps>0$ there exists $\delta>0$ such that if $A\subset H$ is $\mu$-measurable with $\mu(A)<\delta$, then
		\[
		\int_{A}g_i\,d\mu<\eps \quad\textrm{for all }i\in\N.
		\]
	\end{itemize}
\end{definition}

Next we give (special cases of) the Dunford--Pettis theorem,
Mazur's lemma, and Fuglede's lemma, see e.g.
\cite[Theorem 1.38]{AFP},
\cite[Lemma 6.1]{BB},
and \cite[Lemma 2.1]{BB}, respectively.
Let $H\subset X$ be a $\mu$-measurable set.

\begin{theorem}\label{thm:dunford-pettis}
Let $\{g_i\}_{i=1}^{\infty}$ be an equi-integrable sequence of nonnegative functions in $L^1(H)$.
Then there exists a subsequence such that $g_{i_j}\to g$ weakly in $L^1(H)$.
\end{theorem}

\begin{lemma}\label{lem:Mazur lemma}
	Let $\{g_i\}_{i\in\N}$ be a sequence
	with $g_i\to g$ weakly in $L^1(H)$.
	Then there exist convex combinations $\widehat{g}_i:=\sum_{j=i}^{N_i}a_{i,j}g_j$,
	for some $N_i\in\N$,
	such that $\widehat{g}_i\to g$ in $L^1(H)$.
\end{lemma}

\begin{lemma}\label{lem:Fuglede lemma}
Let $\{g_i\}_{i=1}^{\infty}$ be a sequence of functions with $g_i\to g$ in $L^1(H)$.
Then for $1$-a.e. curve $\gamma$ in $H$, we have
\[
\int_{\gamma}g_i\,ds\to \int_{\gamma}g\,ds\quad\textrm{as }i\to\infty.
\]
\end{lemma}

Moreover, we will need the following Vitali-Carath\'eodory theorem;
for a proof see e.g. \cite[p. 108]{HKSTbook}.
The symbol $\Om\subset X$ will always denote an open set.

\begin{theorem}\label{thm:VitaliCar}
Let $\rho\in L^1(\Om)$. Then there exists a sequence $\{\rho_i\}_{i\in \N}$
of lower semicontinuous functions
on $\Om$ such that $\rho\le \rho_{i+1}\le\rho_i$ for all $i\in\N$, and
$\rho_i\to\rho$ in $L^1(\Om)$.
\end{theorem}

\begin{definition}\label{def:upper gradient}
Let $H\subset X$ and $f\colon H\to Y$. We say that a Borel
function $g\colon H\to [0,\infty]$ is an upper gradient of $f$ in $H$ if
\begin{equation}\label{eq:upper gradient inequality}
d_Y(f(\gamma(0)),f(\gamma(\ell_{\gamma})))\le \int_{\gamma}g\,ds
\end{equation}
for every curve $\gamma$ in $H$.
We use the conventions $\infty-\infty=\infty$ and
$(-\infty)-(-\infty)=-\infty$.
If $g\colon H\to [0,\infty]$ is a $\mu$-measurable function
and (\ref{eq:upper gradient inequality}) holds for $1$-a.e. curve in $H$,
we say that $g$ is a $1$-weak upper gradient of $f$ in $H$.
\end{definition}

Let $H\subset X$ be $\mu$-measurable.
We say that $f\in L^1(H;Y)$ if $d_Y(f(\cdot);f(x))\in L^1(H)$ for some $x\in H$.

\begin{definition}
The Newton-Sobolev class $N^{1,1}(H;Y)$ consists of those
functions $f\in L^1(H;Y)$ for which there exists an upper gradient $g\in L^1(H)$.

The Dirichlet space $D^1(H;Y)$ consists of those functions $f\colon H\to Y$
with an upper gradient $g\in L^1(H)$.
\end{definition}

In the classical setting of $X=Y=\R^n$, both spaces equipped with the Lebesgue measure,
and assuming $f$ is  continuous, we have $f\in N^{1,1}(X;Y)$ if and only if
$f\in W^{1,1}(\R^n;\R^n)$, see e.g. \cite[Theorem A.2]{BB}.

To define the class of BV mappings,
we consider the following definitions from Martio \cite{Mar}
who studied BV functions on metric spaces.
Given a family of curves $\Gamma$, we say that a sequence of nonnegative functions $\{\rho_i\}_{i=1}^{\infty}$
is $\AM$-admissible for $\Gamma$ if
\[
\liminf_{i\to\infty}\int_{\gamma}\rho_i\,ds\ge 1\quad \textrm{for all }\gamma\in\Gamma.
\]
Then we let
\[
\AM(\Gamma):=\inf \left\{\liminf_{i\to\infty}\int_X \rho_i\,d\mu\right\},
\]
where the infimum is taken over all $\AM$-admissible sequences $\{\rho_i\}_{i=1}^{\infty}$.
Note that always $\AM(\Gamma)\le \Mod_1(\Gamma)$.

Now we define BV mappings as follows.

\begin{definition}\label{def:BV def}
Given a continuous mapping $f\colon \Om\to Y$, we say that
$f\in D^{\BV}(\Om;Y)$ if there is a sequence of
nonnegative functions $\{g_i\}_{i=1}^{\infty}$ that is bounded in $L^1(\Om)$, and
\begin{equation}\label{eq:AM ae curve}
d_Y(f(\gamma(0)),f(\gamma(\ell_{\gamma})))
\le \liminf_{i\to\infty}\int_{\gamma}g_i\,ds
\end{equation}
for $\AM$-a.e. curve $\gamma$ in $\Om$. If also $f\in L^1(\Om;Y)$, then we say that $f\in \BV(\Om;Y)$.
\end{definition}

In particular, it is enough to have \eqref{eq:AM ae curve} for $1$-a.e. curve in $\Omega$.

Durand-Cartagena--Eriksson-Bique--Korte--Shanmugalingam \cite{DEKS} show that in the case where $f$ is real-valued,
and $\mu$ is a doubling measure that supports a
Poincar\'e inequality, Definition \ref{def:BV def} agrees with Miranda's definition of BV functions
given in \cite{Mir},
which in turn agrees in Euclidean spaces with the classical definition.
Thus it is natural for us to take this as a definition.\\

\emph{Throughout this paper we assume that $(X,d,\mu)$
	and $(Y,d_Y,\nu)$ are metric spaces equipped with Borel regular outer measures $\mu$ and $\nu$,
	such that every ball has finite measure.
}

\section{Covering lemma}\label{sec:covering}

In this section we consider the following covering lemma and its consequences.
The lemma is similar to the Besicovitch covering theorem.

\begin{lemma}\label{lem:N countable collections}
Let $A\subset X$ be bounded and doubling with constant $M$,
and let $\mathcal F$ be a collection of balls $\{B(x,r_x)\}_{x\in A}$
with radius at most $R>0$.
Then there exist finite or countable subcollections $\mathcal G_1,\ldots,\mathcal G_N$,
with $N=M^4$ and $\mathcal G_j=\{B_{j,l}=B(x_{j,l},r_{j,l})\}_{l}$,
such that 
\begin{enumerate}
	\item $A\subset \bigcup_{j=1}^N\bigcup_{l}B_{j,l}$;
	\item If $j\in\{1,\ldots,N\}$ and $l\neq m$, then $x_{j,l}\in X\setminus B_{j,m}$ and $x_{j,m}\in X\setminus B_{j,l}$;
	\item If $j\in\{1,\ldots,N\}$ and $l\neq m$, then $\tfrac 12  B_{j,l}\cap \tfrac 12  B_{j,m}=\emptyset$.
\end{enumerate}
\begin{proof}
Let
\[
A_1:=\left\{x\in A\colon \tfrac 34 R< r_x\le R\right\}.
\]
Recursively, choose points $x\in A_1$ that are at distance at least $\tfrac 34 R$ from each other,
and call this (at most countable) collection $D$.
In the metric space $(A,d)$, at least $\#D$ (cardinality of $D$) balls of radius
$\tfrac 38 R$ are needed to cover $D$, but since $A$ is bounded and doubling, we know that 
finitely many such balls suffice. Thus $D$ is a finite set.
Hence the balls $\{B(x,r_x)\}_{x\in D}$ cover $A_1$.
Suppose a point $y\in A$ is contained in $N_0$ different balls $B(x,3r_x)$, $x\in D$.
Denote these centers $x_1,\ldots,x_{N_0}$.
One needs at least $N_0$ balls in $(A,d)$, of radius $\tfrac {3}{8}R$, to cover
all these centers, and so one needs
at least $N_0$ balls of radius $\tfrac {3}{8}R$ to cover $B(y,3R)\cap A$.
On the other hand, since $A$ is doubling, it is possible to find
$M^4$ balls of radius $\tfrac {3}{8}R$ that
cover $B(y,3R)\cap A$.
Thus $N_0\le M^4=N$.

Pick a maximal disjoint collection $\mathcal G_1^1$ of balls $B(x,r_x)$, $x\in D$.
Recursively, take  a maximal disjoint collection $\mathcal G_{j+1}^1$ of balls $B(x,r_x)$, $x\in D$,
that were not included in any of the collections $\mathcal G_1^{1},\ldots,\mathcal G_j^{1}$.
The collections $\mathcal G_1^1,\ldots,\mathcal G_N^1$ together contain all of the balls $B(x,r_x)$, $x\in D$;
to see this, suppose by contradiction that $B(y,r_y)$, $y\in D$, is a ball not contained in these collections.
Then necessarily $B(y,r_y)$ intersects some ball $B_j\in \mathcal G_j^1$ for every $j=1,\ldots,N$.
Now $y$ is contained in at least $N+1$ balls $B(x,3r_x)$, $x\in D$, contradicting what was proved earlier.

Supposing the collections $\mathcal G_1^n,\ldots,\mathcal G_N^n$ have been defined, define
\[
A_{n+1}:=\left\{x\in A\setminus \Bigg(\bigcup_{k=1}^n\bigcup_{j=1}^{N}\bigcup_{B\in \mathcal G^k_j}B\Bigg)\colon
\left(\frac{3}{4}\right)^{n+1}R< r_x\le \left(\frac{3}{4}\right)^{n}R\right\}.
\]
Then as above, we can extract collections $\mathcal G_1^{n+1},\ldots,\mathcal G_N^{n+1}$
of disjoint balls covering $A_{n+1}$.
Finally define $\mathcal G_j:=\bigcup_{n=1}^{\infty}\mathcal G_j^{n}$, $j=1,\ldots,N$.
We can index each $\mathcal G_j$ as $\{B_{j,l}\}_{l}$.
Clearly (1) is satisfied.
If $j\in\{1,\ldots,N\}$ and $l\neq m$, consider the two balls $B_{j,l}=B(x_{j,l},r_{j,l})$ and $B_{j,m}=B(x_{j,m},r_{j,m})$.
We have $B_{j,l}\in \mathcal G_j^{n_1}$ and $B_{j,m}\in \mathcal G_j^{n_2}$ for some $n_1,n_2\in\N$.
If $n_1=n_2$, then the balls $B_{j,l}$ and $B_{j,m}$ are disjoint and so obviously (2) and (3) hold.
Assume that $n_1<n_2$, the opposite inequality being analogous.
Then by construction $x_{j,m}\notin B_{j,l}$. On the other hand,
$r_{j,m}\le r_{j,l}$, so then also $x_{j,l}\notin B_{j,m}$.
The fact that $x_{j,m}\notin B_{j,l}$ implies that
$B(x_{j,l},\tfrac 12 r_{j,l})$ and $B(x_{j,m},\tfrac 12 r_{j,l})$ are disjoint, and so also
$B(x_{j,l},\tfrac 12 r_{j,l})$ and $B(x_{j,m},\tfrac 12 r_{j,m})$ are disjoint.
\end{proof}
\end{lemma}

Now we show how condition (2) of Lemma \ref{lem:N countable collections} can be used
with quasiconformal mappings.
Recall the definition of $L_f(x,r)$ from the beginning of the introduction.

\begin{lemma}\label{lem:disjoint images}
Consider a collection of balls $\{B_l=B(x_l,r_l)\}_{l=1}^{\infty}$
such that for all $l\neq m$ we have $x_l\notin B_m$
and $x_m\notin B_l$. Also consider an injection $f\colon X\to Y$ such that for all $l\in\N$ there is
$H_l\ge 1$ such that 
\[
B\left(f(x_l),\frac{L_f(x_l,r_l)}{H_l}\right)\subset f(B_l)\quad\textrm{for all }l\in\N.
\]
Then
\[
B\left(f(x_l),\frac{L_f(x_l,r_l)}{2H_l}\right)\cap B\left(f(x_m),\frac{L_f(x_m,r_m)}{2H_m}\right)=\emptyset
\quad\textrm{for all }l\neq m.
\]
\end{lemma}

\begin{proof}
We have $f(x_l)\notin f(B_m)$ and so $f(x_l)\notin B\left(f(x_m),\tfrac{1}{H_m}L_f(x_m,r_m)\right)$.
Analogously, $f(x_m)\notin B\left(f(x_l),\tfrac{1}{H_l}L_f(x_l,r_l)\right)$. Thus
\[
d_Y(f(x_l),f(x_m))\ge \frac{L_f(x_l,r_l)}{2H_l}+\frac{L_f(x_m,r_m)}{2H_m},
\]
and so the conclusion follows.
\end{proof}

\begin{remark}
The previous two lemmas are analogous to Lemma 2.2 and Lemma 2.3 of \cite{BKR}.
However, in the conclusion of Lemma \ref{lem:disjoint images}
we are able to divide by a number comparable to $H_{l}$ (or $H_m$), namely $2H_{l}$,
whereas in \cite[Lemma 2.3]{BKR} it was necessary to divide by a number comparable to $H_l^2$.
In the proof of Theorem \ref{thm:main theorem} we consider large values of $H_l$, and there
this difference becomes crucial.
\end{remark}

\section{Proof of the main theorem}\label{sec:main theorem}

In this section we prove the main theorem, Theorem \ref{thm:main theorem intro}.

We give it in the following more general form.
In particular, note that we consider two measures $\mu$ and $\widetilde{\mu}$
on the space $X$, which may appear odd at first but turns out to be useful
when proving Corollary \ref{cor:weights}.
The idea is simply that $\mu$ might not be asymptotically doubling at \emph{every}
point $x\in \Om\setminus E$, but often a small modification of $\mu$
gives another measure $\widetilde{\mu}$ that satisfies \eqref{eq:two measure doubling}.
The notation $\widetilde{\mu}\ll \mu$ means that $\widetilde{\mu}$ is absolutely continuous
with respect to $\mu$.

By a slight abuse of notation, we denote also the image of a curve $\gamma$ by the same symbol.
Moreover, recall that it is possible to have $h_f<1$ at some points, but we define
$h_f^{\vee}:=\max\{h_f,1\}$ and similarly $H_f^{\vee}(x,r):=\max\{H_f(x,r),1\}$.

\begin{theorem}\label{thm:main theorem}
	Let $\Om\subset X$ be open and bounded,
	let $f\colon \Om\to Y$ be injective and continuous with $\nu(f(\Om))<\infty$,
	and:
	\begin{itemize}
		\item[(1)] Suppose that $(\Om,d)$ is doubling,
		and that there exists another Borel regular outer measure
		$\mu\le \widetilde{\mu}\ll \mu$ on $\Om$ and a $\widetilde{\mu}$-measurable set $E\subset\Om$ such that
		for some constant $C_{\Om}>1$,
		\begin{equation}\label{eq:two measure doubling}
		\limsup_{r\to 0}\frac{\mu(B(x,20r))}{\widetilde{\mu}(B(x,r))}< C_{\Om}\quad\textrm{for every }x\in \Om\setminus E,
		\end{equation}
		and that in
		$\Om\setminus E$ there exist functions $Q(\cdot)\ge 1$ and $R(\cdot)>0$
		that are $\widetilde{\mu}$-measurable and $Q(f^{-1}(\cdot))$ is $\nu$-measurable,
		such that
		\begin{equation}\label{eq:limsup liminf condition}
		\limsup_{r\to 0}\frac{\widetilde{\mu}(B(x,r))}{r^{Q(x)}}<R(x)\liminf_{r\to 0}\frac{\nu(B(f(x),r))}{r^{Q(x)}}
		\quad \textrm{for every }x\in \Om\setminus E.
		\end{equation}
		
		\item[(2)] Suppose also that the exceptional set $E$ satisfies
		\begin{equation}\label{eq:E assumption}
		\Mod_1(\{\gamma\subset \Om\colon \mathcal H^1(f(\gamma\cap E))>0\})=0
		\end{equation}
		and $h_f<\infty$ in $\Om\setminus E$.
		
		\item[(3)] Finally assume that
		\[
		\begin{cases}
		\frac{Q(\cdot)-1}{Q(\cdot)}(R(\cdot)h_f^{\vee}(\cdot)^{Q(\cdot)})^{1/(Q(\cdot)-1)}
		\in L^1(\{Q>1\}\setminus E,\widetilde{\mu});\\
		h_f^{\vee}(x) R(x) \le H<\infty\quad\textrm{for all }x\in \{Q=1\}\setminus E.
		\end{cases}
		\]
	\end{itemize}
	Then $f\in D^{\BV}(\Om;Y)$ and $f\in D^{1}(\{Q>1\}\cup E;Y)$.
\end{theorem}

\begin{proof}
	We divide the proof into two parts.\\
	
	\textbf{Part 1.} Here we show that $f\in D^{\BV}(\Om;Y)$.\\
	
	Fix $0<\delta<1$.
	By the Vitali-Carath\'eodory theorem (Theorem \ref{thm:VitaliCar}),
	we find a lower semicontinuous function $h\ge 0$ on $\Om$ with
	\[
	h\ge \frac{Q(\cdot)-1}{Q(\cdot)}(R(\cdot)h_f^{\vee}(\cdot)^{Q(\cdot)})^{1/(Q(\cdot)-1)}\ch_{\{Q>1\}\setminus E}
	\]
	and
	\[
	\int_{\Om}h\,d\widetilde{\mu}<\int_{\{Q>1\}\setminus E}\frac{Q(\cdot)-1}{Q(\cdot)}
	(R(\cdot)h_f^{\vee}(\cdot)^{Q(\cdot)})^{1/(Q(\cdot)-1)}\,d\widetilde{\mu}+\delta.
	\]
	Define the sets
	\[
	A_1:=\{x\in \Om\colon Q(x)\ge 2\}
	\]
	and
	\begin{equation}\label{eq:def of Ak}
	A_k:=\{x\in \Om\colon 1+1/k\le Q(x)< 1+1/(k-1)\},\quad k=2,3\,\ldots,
	\end{equation}
	and finally $A_{0}:=\{x\in \Om\colon Q(x)=1\}$.
	We have $\bigcup_{k=0}^{\infty}A_k=\Om\setminus E$.
	Since $\widetilde{\mu}$ and $\nu$ are Borel regular outer measures,
	we find open sets $U_k$ with $A_k\subset U_k\subset \Om$
	and open sets $W_k\supset f(A_k)$ such that
	\begin{equation}\label{eq:choice of Uk}
	\int_{U_k}h\,d\widetilde{\mu}< \int_{A_k}h\,d\widetilde{\mu}+
	2^{-k}\delta
	\end{equation}
	and
	\begin{equation}\label{eq:choice of Wk}
	\nu(W_k)< \nu(f(A_k))+2^{-k}\delta.
	\end{equation}
	For every $x\in A_k$ we can choose a radius $0<r_x<\delta$ sufficiently small so that
	\begin{itemize}
		\item $B(x,r_x)\subset U_k$ and
		$B(f(x),L_f(x,r_x))\subset W_k$ (since $f$ is continuous);
		\item by \eqref{eq:two measure doubling},
		\begin{equation}\label{eq:doubling at rx}
		\frac{\mu(B(x,10r))}{\widetilde{\mu}(B(x,\tfrac 12 r))}< C_{\Om};
		\end{equation}
		\item by \eqref{eq:limsup liminf condition}, for some $C(x)>0$,
		\begin{equation}\label{eq:sup and inf}
		\sup_{0<r\le r_x}\frac{\widetilde{\mu}(B(x,r))}{r^{Q(x)}}< R(x)C(x)
		\quad \textrm{and}\quad \inf_{0<r\le L_f(x,r_x)}\frac{\nu(B(f(x),r))}{r^{Q(x)}}>C(x);
		\end{equation}
		\item if $Q(x)>1$ (i.e. if $k\ge 1$), by lower semicontinuity
		$h(y)\ge h(x)/2\ge 0$ for all $y\in B(x,r_x/2)$, and so
		\[
		h(x)\le 2\vint{B(x,r_x/2)}h(y)\,d\widetilde{\mu}(y),
		\]
		whence 
		\begin{equation}\label{eq:Lebesgue point inequality}
		\frac{Q(x)-1}{Q(x)}(R(x)h_f^{\vee}(x)^{Q(x)})^{1/(Q(x)-1)}\le 2\vint{B(x,r_x/2)}h(y)\,d\widetilde{\mu}(y).
		\end{equation}
	\end{itemize}
	Finally since $h_f^{\vee}(x)=\liminf_{r\to 0} H_f^{\vee}(x,r)$,
	we can also choose $r_x$ suitably to have
	\begin{equation}\label{eq:Hf and hf}
	H_f^{\vee}(x,r_x)\le 2h_f^{\vee}(x)
	\end{equation}
	(here we use the fact that $h_f^{\vee}\ge 1>0$).
	For each $k=0,1,\ldots$,
	from this covering $\mathcal G_k:=\{B(x,r_x)\}_{x\in A_k}$,
	by Lemma \ref{lem:N countable collections} we can extract subcoverings
	$\mathcal G_{k,1},\ldots,\mathcal G_{k,N}$,
	\[
	\mathcal G_{k,j}=\{B_{k,j,l}=B(x_{k,j,l},r_{k,j,l})\}_{l},
	\]
	having the good properties given in the Lemma.
	Define
	\[
	\rho:=2\sum_{k=0}^{\infty}\sum_{j=1}^N\sum_{l}
	\frac{L_f(x_{k,j,l},r_{k,j,l})}{r_{k,j,l}}\ch_{2B_{k,j,l}}.
	\]
	Consider a curve $\gamma$ in $\Om$ with $\diam\gamma \ge\delta$.
	If $\gamma$ intersects $B_{k,j,l}$, then $\mathcal H^1(\gamma\cap 2B_{k,j,l})\ge r_{k,j,l}$.
	Thus for $1$-a.e. curve $\gamma$ in $\Om$, we have
	\begin{equation}\label{eq:upper gradient property proved}
	\int_{\gamma}\rho\,ds\ge 2\sum_{B_{k,j,l}\cap \gamma\neq\emptyset}L_f(x_{k,j,l},r_{k,j,l})
	\ge  \sum_{B_{k,j,l}\cap \gamma\neq\emptyset}\diam (f(B_{k,j,l}))
	\ge d_Y(f(\gamma(0)),f(\gamma(\ell_{\gamma}))),
	\end{equation}
	where the last inequality holds since the balls $B_{k,j,l}$ satisfying $B_{k,j,l}\cap \gamma\neq\emptyset$
	cover $\gamma\setminus E$ and so the sets $f(B_{k,j,l})$ cover
	$\mathcal H^1$-almost all of $f(\gamma)$
	by \eqref{eq:E assumption}.
	We also define
	\begin{equation}\label{eq:definition of g}
	g:=2\sum_{k=1}^{\infty}\sum_{j=1}^N\sum_{l}\frac{L_f(x_{k,j,l},r_{k,j,l})}{r_{k,j,l}}\ch_{2B_{k,j,l}};
	\end{equation}
	that is, the same as $\rho$ but with the summation starting from $k=1$ instead of $k=0$.
	For $1$-a.e. curve $\gamma$ in $\{Q>1\}\cup E$ with $\diam\gamma \ge \delta$, we similarly obtain
	\begin{equation}\label{eq:upper gradient property proved 2}
	\int_{\gamma}g\,ds\ge d_Y(f(\gamma(0)),f(\gamma(\ell_{\gamma}))).
	\end{equation}
	Note that for each ball $B_{k,j,l}$, we have
	\[
	B\left(f(x_{k,j,l}),\frac{L_f(x_{k,j,l},r_{k,j,l})}{H_f^{\vee}(x_{k,j,l},r_{k,j,l})}\right)
	\subset B\left(f(x_{k,j,l}),l_f(x_{k,j,l},r_{k,j,l})\right)\\
	\subset f(B_{k,j,l}).
	\]
	Now, for any fixed $k\in\N$ and $j\in\{1,\ldots,N\}$, Lemma \ref{lem:disjoint images} gives
	\begin{equation}\label{eq:disjoint images}
	B\left(f(x_{k,j,l}),\frac{L_f(x_{k,j,l},r_{k,j,l})}{2H_f^{\vee}(x_{k,j,l},r_{k,j,l})}\right)\cap B\left(f(x_{k,j,m}),
	\frac{L_f(x_{k,j,m},r_{k,j,m})}{2H_f^{\vee}(x_{k,j,m},r_{k,j,m})}\right)
	=\emptyset
	\quad\textrm{for all }l\neq m.
	\end{equation}
	Denote $C_{k,j,l}:=C(x_{k,j,l})$, $R_{k,j,l}:=R(x_{k,j,l})$, and $Q_{k,j,l}:=Q(x_{k,j,l})$.
	By \eqref{eq:sup and inf}, when $k\in\N$ we get
	\begin{equation}\label{eq:measure of B divided by rq}
	\begin{split}
	\left(C_{k,j,l}^{-1/Q_{k,j,l}}\frac{\widetilde{\mu}(\tfrac 12 B_{k,j,l})}{r_{k,j,l}}\right)^{Q_{k,j,l}/(Q_{k,j,l}-1)}
	&=C_{k,j,l}^{-1/(Q_{k,j,l}-1)}\widetilde{\mu}(\tfrac 12 B_{k,j,l})\frac{\widetilde{\mu}(\tfrac 12 B_{k,j,l})^{1/(Q_{k,j,l}-1)}}{(r_{k,j,l})^{Q_{k,j,l}/(Q_{k,j,l}-1)}}\\
	&\le R_{k,j,l}^{1/(Q_{k,j,l}-1)}\widetilde{\mu}(\tfrac 12 B_{k,j,l}).
	\end{split}
	\end{equation}
	For every $k\in\N$, abbreviating $\sum_{j=1}^N \sum_{l}$ by $\sum_{j,l}$, by
	\eqref{eq:doubling at rx} we have
	\begin{align*}
	&\sum_{j,l}\frac{L_f(x_{k,j,l},r_{k,j,l})}{r_{k,j,l}}\mu(10B_{k,j,l})\\
	&\qquad\le C_{\Om}\sum_{j,l}\frac{L_f(x_{k,j,l},r_{k,j,l})}{h_f^{\vee}(x_{k,j,l})}h_f^{\vee}(x_{k,j,l})
	\frac{\widetilde{\mu}(\tfrac 12 B_{k,j,l})}{r_{k,j,l}}\\
	&\qquad\le 4C_{\Om} \sum_{j,l}C_{k,j,l}^{1/Q_{k,j,l}}\frac{L_f(x_{k,j,l},r_{k,j,l})}{2H_f^{\vee}(x_{k,j,l},r_{k,j,l})}
	h_f^{\vee}(x_{k,j,l}) C_{k,j,l}^{-1/Q_{k,j,l}} 
	\frac{\widetilde{\mu}(\tfrac 12 B_{k,j,l})}{r_{k,j,l}}\quad\textrm{by }\eqref{eq:Hf and hf}\\
	&\qquad\le 4C_{\Om}\sum_{j,l}C_{k,j,l}\left(\frac{L_f(x_{k,j,l},r_{k,j,l})}{2H_f^{\vee}(x_{k,j,l},r_{k,j,l})}\right)^{Q_{k,j,l}}\\
	&\qquad\qquad +4C_{\Om} \sum_{j,l} \frac{Q_{k,j,l}-1}{Q_{k,j,l}}h_f^{\vee}(x_{k,j,l})^{Q_{k,j,l}/(Q_{k,j,l}-1)}
	R_{k,j,l}^{1/(Q_{k,j,l}-1)}\widetilde{\mu}(\tfrac 12 B_{k,j,l}),
	\end{align*}
	where we used Young's inequality, estimating simply  $1/Q_{k,j,l}\le 1$ for the first term,
	and also \eqref{eq:measure of B divided by rq}.
	Using the second inequality of \eqref{eq:sup and inf} as well as \eqref{eq:Lebesgue point inequality},
	we estimate further
	\begin{align*}
	&\sum_{j,l}\frac{L_f(x_{k,j,l},r_{k,j,l})}{r_{k,j,l}}\mu(10B_{k,j,l})\\
	&\qquad\le 4C_{\Om}\sum_{j,l}\nu\Big(B\left(f(x_{k,j,l}), \frac{L_f(x_{k,j,l},r_{k,j,l})}{2H_f^{\vee}(x_{k,j,l},r_{k,j,l})}\right)\Big)
	+8C_{\Om}\sum_{j,l}\int_{\tfrac 12 B_{k,j,l}}h\,d\widetilde{\mu}\\
	&\qquad\le 4C_{\Om} N\nu(W_k)
	+8C_{\Om}\sum_{j,l}\int_{\tfrac 12 B_{k,j,l}}h\,d\widetilde{\mu}\quad\textrm{by }
	\eqref{eq:disjoint images}\\
	&\qquad\le  4C_{\Om} N(\nu(f(A_k))+2^{-k}\delta)
	+8C_{\Om} N \int_{U_k}h\,d\widetilde{\mu}\quad\textrm{by }\eqref{eq:choice of Wk}\\
	&\qquad\le  4C_{\Om} N(\nu(f(A_k))+2^{-k}\delta)
	+8C_{\Om} N \left(\int_{A_k} h\,d\widetilde{\mu}+2^{-k}\delta\right)
	\quad\textrm{by }\eqref{eq:choice of Uk}.
	\end{align*}
	We record this result: for all $k\in\N$,
	\begin{equation}\label{eq:big sum estimate}
	\begin{split}
	\sum_{j,l}\frac{L_f(x_{k,j,l},r_{k,j,l})}{r_{k,j,l}}\mu(10B_{k,j,l})\le  4C_{\Om} N(\nu(f(A_k))+2^{-k}\delta)
	+8C_{\Om} N \left(\int_{A_k}h\,d\widetilde{\mu}+2^{-k}\delta\right).
	\end{split}
	\end{equation}
	For $k=0$, by \eqref{eq:doubling at rx} and \eqref{eq:Hf and hf} we get
	\begin{equation}\label{eq:big sum estimate k0}
	\begin{split}
	&\sum_{j,l}\frac{L_f(x_{0,j,l},r_{0,j,l})}{r_{0,j,l}}\mu(2B_{0,j,l})\\
	&\quad \le 4C_{\Om}\sum_{j,l}C_{0,j,l}\frac{L_f(x_{0,j,l},r_{0,j,l})}{2H_f^{\vee}(x_{0,j,l},r_{0,j,l})}
	h_f^{\vee}(x_{0,j,l}) C_{0,j,l}^{-1} 
	\frac{\widetilde{\mu}(\tfrac 12 B_{0,j,l})}{r_{0,j,l}}\\
	&\quad \le 4C_{\Om}\sum_{j,l}\nu\left(B\left(f(x_{0,j,l}), \frac{L_f(x_{0,j,l},r_{0,j,l})}{2H_f^{\vee}(x_{0,j,l},r_{0,j,l}))}\right)\right)h_f^{\vee}(x_{0,j,l})R_{0,j,l}
	\quad\textrm{by }\eqref{eq:sup and inf}\\
	&\quad \le 4 C_{\Om} N\nu(W_{0})\cdot H\quad\textrm{by }\eqref{eq:disjoint images}
	\textrm{ and since }h_f^{\vee}(x)R(x)\le H\textrm{ for all }x\in \{Q=1\}\setminus E\\
	&\quad \le 4H C_{\Om} N(\nu(f(A_{0}))+\delta)\quad\textrm{by }\eqref{eq:choice of Wk}.
	\end{split}
	\end{equation}
	Combining this with \eqref{eq:big sum estimate},
	and noting that the sets $f(A_k)$ are $\nu$-measurable
	by the measurability of $Q(f^{-1}(\cdot))$,
	we get
	(the fact that the left-hand side of \eqref{eq:big sum estimate} has the balls
	$10B_{k,j,l}$ instead of $2B_{k,j,l}$ will be useful later)
	\begin{align*}
	\Vert \rho\Vert_{L^1(\Om)}
	&\le 2\sum_{k=0}^{\infty}\sum_{j,l}\frac{L_f(x_{k,j,l},r_{k,j,l})}{r_{k,j,l}}\mu(2B_{k,j,l})\\
	&\le 8H C_{\Om} N(\nu(f(A_{0}))+\delta)
	+8C_{\Om} N(\nu(f(\Om))+\delta)
	+16C_{\Om} N \left(\int_{\Om}h\,d\widetilde{\mu}+\delta\right).
	\end{align*}
Recall that $\delta>0$ was the upper bound for the radii in the coverings. 
Now we can choose functions $\rho$ with the choices $\delta=1/i$,
to get a sequence $\{\rho_i\}_{i=1}^{\infty}$ that is bounded in $L^1(\Om)$.

By \eqref{eq:upper gradient property proved}, the pair $(f,\rho_i)$ satisfies the upper 
gradient inequality for $1$-a.e.
curve $\gamma$ in $\Om$ with $\diam\gamma\ge 1/i$. Denote the exceptional family by
$\Gamma_i$, and $\Gamma:=\bigcup_{i=1}^{\infty}\Gamma_i$.
If $\gamma\notin \Gamma$ is a nonconstant curve in $\Om$,
then by \eqref{eq:upper gradient property proved} we have
\[
d_Y(f(\gamma(0)),f(\gamma(\ell_{\gamma})))\le
\int_{\gamma}\rho_i\,ds
\]
for every $i\in\N$ for which $\diam\gamma\ge 1/i$.
Thus
\[
d_Y(f(\gamma(0)),f(\gamma(\ell_{\gamma})))\le
\liminf_{i\to\infty}\int_{\gamma}\rho_i\,ds
\]
and so $f\in D^{\BV}(\Om;Y)$ by Definition \ref{def:BV def}.\\

\textbf{Part 2.} Now we show that $f\in D^{1}(\{Q>1\}\cup E;Y)$.

Recall the definition of $g$ from \eqref{eq:definition of g}.
As above with $\rho$, we can choose functions $g$ with the choices $\delta=1/i$,
to get a sequence $\{g_i\}_{i=1}^{\infty}$ that is bounded in $L^1(\Om)$.
We will show that $\{g_i\}_{i=1}^{\infty}$ is an equi-integrable sequence in $L^1(\Om)$.

The first condition of Definition \ref{def:equiintegrability}
holds automatically since $\Om$ as a bounded set has finite $\mu$-measure.
We check the second condition. Suppose by contradiction that
by passing to a subsequence of $\{g_i\}_{i=1}^{\infty}$ (not relabeled),
we find $0<\eps<1$ and
a sequence of sets $H_i\subset \Om$ such that $\mu(H_i)\to 0$ and
\[
\int_{H_i} g_{i}\,d\mu\ge \eps\quad\textrm{for all }i\in\N.
\]
By assumption $\widetilde{\mu}\ll  \mu$, so also $\widetilde{\mu}(H_i)\to 0$.
Now the balls $B_{k,j,l}=B(x_{k,j,l},r_{k,j,l})$ depend also on $i\in\N$, so we denote them
$B_{i,k,j,l}=B(x_{i,k,j,l},r_{i,k,j,l})$.
Thus we have
\begin{equation}\label{eq:equiintegrability first assumption}
\sum_{k=1}^{\infty}\sum_{j,l}\frac{L_f(x_{i,k,j,l},r_{i,k,j,l})}{r_{i,k,j,l}}\mu(2B_{i,k,j,l}\cap H_i)
\ge \frac{\eps}{2}\quad\textrm{for all }i\in\N.
\end{equation}
Note that since $f$ is injective,
\[
f\left(\bigcup_{k=K}^{\infty}A_k\right)\searrow \emptyset\quad \textrm{as }K\to\infty.
\]
Thus
\[
\nu\Big(f\left(\bigcup_{k=K}^{\infty}A_k\right)\Big)
\to 0\quad\textrm{as }K\to\infty,
\]
and we also have
\[
\int_{\bigcup_{k=K}^{\infty}A_k}h\,d\widetilde{\mu}\to 0\quad\textrm{as }K\to\infty.
\]
Thus by \eqref{eq:big sum estimate}, choosing and fixing a sufficiently large $K\in\N$, we have
\begin{align*}
&\sum_{k=K}^{\infty}\sum_{j,l}\frac{L_f(x_{i,k,j,l},r_{i,k,j,l})}{r_{i,k,j,l}}\mu(2B_{i,k,j,l})\\
&\qquad\le 4C_{\Om} N\big(\nu\Big(f\left(\bigcup_{k=K}^{\infty}A_k\right)\Big)+2^{-K+1}\big)
 +8C_{\Om} N \left(\int_{\bigcup_{k=K}^{\infty}A_k}h\,d\widetilde{\mu}+2^{-K+1}\right)\\
&\qquad<\frac{\eps}{4}
\end{align*}
for all $i\in\N$.
Thus by \eqref{eq:equiintegrability first assumption}, we have in fact
\[
\sum_{k=1}^{K}\sum_{j,l}\frac{L_f(x_{i,k,j,l},r_{i,k,j,l})}{r_{i,k,j,l}}\mu(2B_{i,k,j,l}\cap H_i)
\ge \frac{\eps}{4}\quad\textrm{for all }i\in\N.
\]
Passing to another subsequence (not relabeled) we can assume that
for some fixed $k\in\{1,\ldots,K\}$, we have
\[
\sum_{j,l}\frac{L_f(x_{i,k,j,l},r_{i,k,j,l})}{r_{i,k,j,l}}
\mu(2B_{i,k,j,l}\cap H_i)\ge \frac{\eps}{4K}\quad\textrm{for all }i\in\N.
\]
Passing to yet another sequence (not relabeled), we can assume that for some fixed $j\in\{1,\ldots,N\}$,
\[
 \sum_{l}\frac{L_f(x_{i,k,j,l},r_{i,k,j,l})}{r_{i,k,j,l}}
\mu(2B_{i,k,j,l}\cap H_i)\ge \frac{\eps}{4KN}\quad\textrm{for all }i\in\N.
\]
The indices $k$ and $j$ are now fixed, so we drop them from the notation.
Thus we write
\begin{equation}\label{eq:equintegrability assumption}
 \sum_{l}\frac{L_f(x_{i,l},r_{i,l})}{r_{i,l}}
\mu(2B_{i,l}\cap H_i)\ge \frac{\eps}{4KN}\quad\textrm{for all }i\in\N.
\end{equation}
Choose $M$ to be the following (very large) number:
\begin{equation}\label{eq:M choice}
M:=
\frac{10KN}{\eps}\left[4C_{\Om} N(\nu(f(\Om))+1)
+8C_{\Om} N \left(\int_{\Om}h\,d\widetilde{\mu}+1\right)\right],
\end{equation}
and consider two sets of indices $I_1^{i}$ and $I_2^{i}$ as follows:
for $l\in I_1^{i}$, we have
\[
\frac{\widetilde{\mu}(2 B_{i,l}\cap H_i)}{\mu(10B_{i,l})}\le \frac{1}{M}.
\]
Then let $I_2^{i}$ consist of the remaining indices.
By the 5-covering lemma (see e.g. \cite[p. 60]{HKSTbook}), we find $J^i\subset I_2^i$ such that
the balls $2B_{i,l}$, $l\in J^i$, are pairwise disjoint and
\[
\bigcup_{l\in I_2^i}2B_{i,l}
\subset \bigcup_{l\in J^i}10 B_{i,l}.
\]
Thus
\begin{equation}\label{eq:measure of Itwo balls}
\mu\left(\bigcup_{l\in I_2^i}2B_{i,l}\right)
\le \sum_{l\in  J^i} \mu(10B_{i,l})
\le MC_{\Om}\sum_{l\in J^i} \widetilde{\mu}(2  B_{i,l}\cap H_i)
\le MC_{\Om}\widetilde{\mu}(H_i).
\end{equation}
Now
\begin{equation}\label{eq:split into two sums}
\begin{split}
&\sum_{l}\frac{L_f(x_{i,l},r_{i,l})}{r_{i,l}}
\mu(2B_{i,l}\cap H_i)\\
&\qquad=  \sum_{l\in I_1^i} \frac{L_f(x_{i,l},r_{i,l})}{r_{i,l}}
\mu(2B_{i,l}\cap H_i)
+ \sum_{l\in I_2^i} \frac{L_f(x_{i,l},r_{i,l})}{r_{i,l}}
\mu(2B_{i,l}\cap H_i).
\end{split}
\end{equation}
For the first term, we can make the (rough) estimate
\begin{align*}
&\sum_{l\in I_1^i}\frac{L_f(x_{i,l},r_{i,l})}{r_{i,l}}\mu(2B_{i,l}\cap H_i)\\
&\qquad \le\sum_{l\in I_1^i}\frac{L_f(x_{i,l},r_{i,l})}{r_{i,l}}\widetilde{\mu}(2B_{i,l}\cap H_i)\\
&\qquad \le \frac{1}{M}\sum_{l=1}^{\infty}\frac{L_f(x_{i,l},r_{i,l})}{r_{i,l}}\mu(10B_{i,l})\\
&\qquad \le\frac{1}{M}\left[4C_{\Om} N(\nu(f(\Om))+1/i)
+8C_{\Om} N \left(\int_{\Om}h\,d\widetilde{\mu}+1/i\right)\right]
\end{align*}
by \eqref{eq:big sum estimate}.
For the second term, we use the generalized Young's inequality:
if $1/p+1/q=1$ and $1+1/K\le p<\infty$, then $q\le K+1$, and so
by the usual Young's inequality we have for any $a,b\ge 0$ and any choice of $0<\kappa<1$ that
\begin{equation}\label{eq:generalized Young}
ab=\kappa^{1/p} a \kappa^{-1/p} b
\le \frac{1}{p}\kappa a^p + \frac{1}{q}\kappa^{-q/p} b^q
\le  \kappa a^p+\frac{1}{q}\kappa^{-K}b^q.
\end{equation}
We choose
\begin{equation}\label{eq:kappa definition}
\kappa:=\frac{\eps}{10KN} \frac{1}{4 C_{\Om} (\nu(f(\Om))+1)}.
\end{equation}
Now by \eqref{eq:doubling at rx} we get
\begin{align*}
&\sum_{l\in I_2^i} \frac{L_f(x_{i,l},r_{i,l})}{r_{i,l}}\mu(2B_{i,l}\cap H_i)\\
&\qquad\le C_{\Om}\sum_{l\in I_2^i}\frac{L_f(x_{i,l},r_{i,l})}{h_f^{\vee}(x_{i,l})}h_f^{\vee}(x_{i,l})
\frac{\widetilde{\mu}(\tfrac 12 B_{i,l})}{r_{i,l}}\\
&\qquad\le 4C_{\Om} \sum_{l\in I_2^i}C_{i,l}^{1/Q_{i,l}}\frac{L_f(x_{i,l},r_{i,l})}{2H_f^{\vee}(x_{i,l},r_{i,l})}
h_f^{\vee}(x_{i,l})C_{i,l}^{-1/Q_{i,l}}
\frac{\widetilde{\mu}(\tfrac 12 B_{i,l})}{r_{i,l}}\quad\textrm{by }\eqref{eq:Hf and hf}\\
&\qquad\le 4 C_{\Om} \kappa \sum_{l\in I_2^i}C_{i,l}\left(\frac{L_f(x_{i,l},r_{i,l})}{2H_f^{\vee}(x_{i,l},r_{i,l})}\right)^{Q_{i,l}}\\
&\qquad\qquad +4C_{\Om}\kappa^{-K}\sum_{l\in I_2^i}\frac{Q_{i,l}-1}{Q_{i,l}}
R_{i,l}^{1/(Q_{i,l}-1)}
h_f^{\vee}(x_{i,l})^{Q_{i,l}/(Q_{i,l}-1)}\widetilde{\mu}(\tfrac 12 B_{i,l})
\end{align*}
by the generalized Young's inequality \eqref{eq:generalized Young} and
\eqref{eq:measure of B divided by rq}; note that $Q_{i,l}\ge 1+1/K$ by the definition of the sets
$A_k$ in \eqref{eq:def of Ak}.
Using the second inequality of \eqref{eq:sup and inf} as well as \eqref{eq:Lebesgue point inequality},
we estimate further
\begin{align*}
&\sum_{l\in I_2^i} \frac{L_f(x_{i,l},r_{i,l})}{r_{i,l}}\mu(2B_{i,l}\cap H_i)\\
&\qquad\le 4 \kappa C_{\Om}  \sum_{l\in I_2^i}\nu\Big(B\left(f(x_{i,l}), \frac{L_f(x_{i,l},r_{i,l})}{2H_f^{\vee}(x_{i,l},r_{i,l}))}\right)\Big)+8\kappa^{-K}C_{\Om}\sum_{l\in I_2^i}\int_{\tfrac 12 B_{i,l}}h\,d\widetilde{\mu}\\
&\qquad\le 4 \kappa C_{\Om} \nu(W_k)
+8\kappa^{-K}C_{\Om}\sum_{l\in I_2^i}\int_{\tfrac 12 B_{i,l}}h\,d\widetilde{\mu}\quad\textrm{by }
\eqref{eq:disjoint images}\\
&\qquad\le  4 \kappa C_{\Om} (\nu(f(A_k))+1/i)
+8\kappa^{-K}C_{\Om} \int_{\bigcup_{l\in I_2^i}2B_{i,l}}h\,d\widetilde{\mu}\quad\textrm{by }\eqref{eq:choice of Wk}\\
&\qquad\le  \frac{\eps}{10KN}
+8\kappa^{-K}C_{\Om} \int_{\bigcup_{l\in I_2^i}2B_{i,l}}h\,d\widetilde{\mu}
\end{align*}
by \eqref{eq:kappa definition}.

In total, for \eqref{eq:split into two sums} we have
\begin{align*}
&\sum_{l=1}^{\infty}\frac{L_f(x_{i,l},r_{i,l})}{r_{i,l}}
\mu(2B_{i,l}\cap H_i)\\
&\qquad\le \frac{1}{M}\left[4C_{\Om} N(\nu(f(\Om))+1/i)
+8C_{\Om} N \left(\int_{\Om}h\,d\widetilde{\mu}+1/i\right)\right]\\
&\qquad\qquad +\frac{\eps}{10KN}
+8\kappa^{-K}C_{\Om} \int_{\bigcup_{l\in I_2^i}2B_{i,l}}h\,d\widetilde{\mu}\\
&\qquad\le \frac{\eps}{10KN}+\frac{\eps}{10KN}+8\kappa^{-K}C_{\Om} \int_{\bigcup_{l\in I_2^i}2B_{i,l}}h\,d\widetilde{\mu}\quad\textrm{by }\eqref{eq:M choice}\\
&\qquad\to \frac{2\eps}{10KN}+0\quad\textrm{as }i\to\infty
\end{align*}
by the absolute continuity of the integral, since by \eqref{eq:measure of Itwo balls}
we had $\mu\left(\bigcup_{l\in I_2^i}2B_{i,l}\right)\le MC_{\Om}\widetilde{\mu}(H_i)\to 0$
and then also $\widetilde{\mu}\left(\bigcup_{l\in I_2^i}2B_{i,l}\right)\to 0$.
But this contradicts \eqref{eq:equintegrability assumption}.
Thus $\{g_i\}_{i=1}^{\infty}$ is an equi-integrable sequence in $L^1(\Om)$.

Then we can use the Dunford-Pettis theorem (Theorem \ref{thm:dunford-pettis})
to conclude that passing to a subsequence (not relabeled) we find $g\in L^1(\Om)$ such that
$g_i\to g$ weakly in $L^1(\Om)$.
Recall from \eqref{eq:upper gradient property proved 2} that the pair $(f,g_i)$ satisfies the
upper gradient inequality for $1$-a.e.
curve $\gamma$ in $\{Q>1\}\cup E$ with $\diam\gamma\ge 1/i$.
By Mazur's lemma (Lemma \ref{lem:Mazur lemma}),
we find convex combinations $\widehat{g}_i:=\sum_{j=i}^{N_i}a_{i,j} g_j$
such that $\widehat{g}_i\to g$ in $L^1(\Om)$ for some $g\in L^1(\Om)$,
and thus also $\widehat{g}_i\to g$ in $L^1(\{Q>1\}\cup E)$.
Each pair $(f,\widehat{g}_i)$ satisfies the upper gradient inequality for $1$-a.e.
curve $\gamma$ in $\{Q>1\}\cup E$ with $\diam\gamma\ge 1/i$.
By Fuglede's lemma (Lemma \ref{lem:Fuglede lemma}),
the pair $(f,g)$ satisfies the upper gradient inequality for $1$-a.e.
curve $\gamma$ in $\{Q>1\}\cup E$ with $\diam\gamma\ge 1/i$ for every $i\in\N$, and thus in fact
for $1$-a.e. nonconstant
curve $\gamma$ in $\{Q>1\}\cup E$.
Thus $g\in L^1(\{Q>1\}\cup E)$ is a $1$-weak upper gradient of $f$ in $\{Q>1\}\cup E$ and so $f\in D^1(\{Q>1\}\cup E;Y)$.
\end{proof}

\section{The exceptional set $E$}\label{sec:exceptional set}

A relevant question concerning Theorem \ref{thm:main theorem} is, what kind of exceptional
sets $E$ satisfy condition \eqref{eq:E assumption} for all continuous mappings
$f\colon \Om\to Y$? In Balogh--Koskela--Rogovin \cite{BKR} and Williams \cite{Wi},
it is assumed that $E$ is $\sigma$-finite with respect
to the $Q-1$-dimensional Hausdorff measure. Since we do not assume the space to be
Ahlfors-regular, we instead consider a codimension $1$ Hausdorff measure.

\begin{definition}\label{def:centered codimension one measure}
For any set $A\subset X$ and $0<R<\infty$, the restricted doubled Hausdorff content
of codimension $1$ is defined by
\[
\widehat{\mathcal{H}}_{R}(A):=\inf\left\{ \sum_{j}
\frac{\mu(B(x_{j},2r_{j}))}{r_{j}}\colon A\subset\bigcup_{j}B(x_{j},r_{j}),\,r_{j}\le R\right\},
\]
where $j$ runs over a finite or countable index set.
The doubled codimension $1$ Hausdorff measure of $A\subset X$ is then defined by
\[
\widehat{\mathcal{H}}(A):=\lim_{R\rightarrow 0}\widehat{\mathcal{H}}_{R}(A).
\]
\end{definition}

Note that our definition is slightly different compared to the usual definition of codimension $1$
Hausdorff measure $\mathcal H$ used in the literature.
The difference is that in the sum we consider balls dilated by
factor $2$, hence the word ``doubled''.
Despite the somewhat peculiar definition, $\widehat{\mathcal{H}}$ is often easy to use,
see Corollary \ref{cor:Linfinity}.
Moreover, if $\mu$ is a doubling measure, then $\mathcal H$ and $\widehat{\mathcal{H}}$ are comparable,
and in an Ahlfors $Q$-regular space both are comparable to $\mathcal H^{Q-1}$.

The following lemma and its proof are similar to \cite[Lemma 3.5]{BKR}.

\begin{lemma}\label{lem:image Lebesgue measure zero}
	Suppose $H\subset X$ and that $f\colon H\to Y$ is continuous. Suppose $E\subset H$ has
	$\sigma$-finite $\widehat{\mathcal{H}}$-measure.
	Then $\Mod_1(\Gamma)=0$ for
	\[
	\Gamma:=\{\gamma\subset H\colon \mathcal H^1(f(\gamma\cap E))>0\}.
	\]
\end{lemma}
\begin{proof}
	\phantom\\
	
	\textbf{Step 1.}
	First assume that $\widehat{\mathcal{H}}(E)<\infty$.
	Let $0<\eps<1$ and
	\[
	\Gamma_{\eps}:=\{\gamma\subset H\colon \ell_{\gamma}>\eps
	\ \textrm{ and }\ \,\mathcal H^1(f(\gamma\cap E))>\eps\}.
	\] 
	For every $l\in\N$ with $1/l<\eps/2$,
	we choose a covering $\{B(x_{l,k},r_{l,k})\}_{k}$ of $E$ with $r_{l,k}\le 1/l$ and
	\[
	\sum_{k}\frac{\mu(B(x_{l,k},2r_{l,k}))}{r_{l,k}}\le \widehat{\mathcal{H}}(E)+1.
	\]
	Define the functions
	\[
	\rho_l:=\sum_{k}\frac{\ch_{B(x_{l,k},2r_{l,k})}}{r_{l,k}},\quad l\in\N.
	\]
	For every $l\in\N$, we have
	\[
	\Vert \rho_l\Vert_{L^1(X)}
	\le \sum_{k}\frac{\mu(B(x_{l,k},2r_{l,k}))}{r_{l,k}}
	\le \widehat{\mathcal{H}}(E)+1.
	\]
	Let $\gamma\in\Gamma_{\eps}$.
	Note that $f$ is continuous on $H$ and thus uniformly continuous in the compact set $\gamma$.
	Since $\mathcal H^1(f(\gamma\cap E))>\eps$,
	also $\mathcal H^1_{1/m}(f(\gamma\cap E))>\eps$ for all sufficiently large $m\in\N$.
	For all such $m$,
	there necessarily exist points $y_1,\ldots,y_m\in f(\gamma\cap E)$ such that
	\[
	d_Y(y_j,y_k)\ge \frac{\eps}{2m},\quad j\neq k.
	\]
	By the uniform continuity, there exists $l$ depending on $m$ such that
	$d_Y(f(x),f(y))< \eps/(2m)$ for all $x,y\in\gamma$ with $d(x,y)<2/l$.
	Thus there are at least $m$ different balls $B(x_{l,k},r_{l,k})$ that have nonempty intersection
	with $\gamma$. Then $\gamma$ intersects each ball $B(x_{l,k},2r_{l,k})$ at least for the length $r_{l,k}$. Hence
	\[
	\int_{\gamma}\rho_l\,ds\ge m.
	\]
	Thus
	\[
	\Mod_1(\Gamma_{\eps})\le \frac{1}{m}\Vert \rho_l\Vert_{L^1(X)}
	\le \frac{1}{m}\left(\widehat{\mathcal{H}}(E)+1\right).
	\]
	Letting $m\to\infty$, we get $\Mod_1(\Gamma_{\eps})=0$.
	Note that $\Gamma=\bigcup_{j=1}^{\infty}\Gamma_{1/j}$ and so
	by the subadditivity of the $1$-modulus,
	\[
	\Mod_1(\Gamma)\le \sum_{j=1}^{\infty}\Mod_1(\Gamma_{1/j})=0.
	\]
	
	\textbf{Step 2.}
	In the general case,
	$E$ has $\sigma$-finite $\widehat{\mathcal{H}}$-measure, so
	we can write $E=\bigcup_{i=1}^{\infty}E_i$ with $\widehat{\mathcal{H}}(E_i)<\infty$.
	Now $\Gamma=\bigcup_{i=1}^{\infty}\Gamma_i$ with
	\[
	\Gamma_i:=\{\gamma\subset H\colon \mathcal H^1(f(\gamma\cap E_i))>0\}.
	\]
	By Step 1, we have $\Mod_1(\Gamma_i)=0$ and so by using subadditivity again, we get
	$\Mod_1(\Gamma)=0$.
\end{proof}

Now we give the following sufficient condition for a set $E$ to satisfy
condition
\eqref{eq:E assumption} for all continuous mappings
$f\colon \Om\to Y$.

\begin{proposition}\label{prop:exceptional set E}
Suppose $f\colon \Om\to Y$ is continuous. Suppose $E=E_1\cup E_2$ is a subset of $\Om$ such that
$E_1$ is an at most countable set, and $E_2$ is $\sigma$-finite
with respect to $\widehat{\mathcal{H}}$.
Then $\Mod_1(\{\gamma\subset \Om\colon \mathcal H^1(f(\gamma\cap E))>0\})=0$.
\end{proposition}
\begin{proof}
For every curve $\gamma\subset\Om$, the set $f(\gamma\cap E_1)$ has cardinality at most that
of $\gamma\cap E_1$, and so $\mathcal H^1(f(\gamma\cap E_1))=0$. Thus the result
follows from Lemma \ref{lem:image Lebesgue measure zero}.
\end{proof}

\begin{proof}[Proof of Theorem \ref{thm:main theorem intro}]
	From the definition of $\widehat{\mathcal H}$ it is straightforward to check that a set with
	finite $\widehat{\mathcal H}$-measure has zero $\mu$-measure, and so $E$ is $\mu$-measurable.
	Note that $Q(f^{-1}(\cdot))$
	is a $\nu$-measurable function on $Y$ since $f$
	as a continuous mapping maps Borel sets to
	\emph{analytic sets}, see \cite[Theorem 14.2, Theorem 21.10]{Kech}.
The theorem now follows from Theorem \ref{thm:main theorem} together with
Proposition \ref{prop:exceptional set E}.
\end{proof}

\begin{remark}\label{rem:on the exceptional set E}
A countable set is always $\sigma$-finite with respect to the $Q-1$-dimensional Hausdorff measure,
when $Q\ge 1$. Thus when considering this measure, the set $E_1$ could always be included
in the set $E_2$ in Proposition \ref{prop:exceptional set E}, but in our setting it can
happen that even a single point has infinite $\widehat{\mathcal H}$-measure,
see Example \ref{ex:bowtie} below.

In an Ahlfors $Q$-regular space, the $Q-1$-dimensional Hausdorff measure is clearly comparable
to $\widehat{\mathcal H}$ (and also comparable to the usual codimension $1$ Hausdorff measure
$\mathcal H$).
Thus in an Ahlfors $Q$-regular space our condition on $E$ reduces to that required in
\cite{BKR,Wi}.

Concerning the assumption $\nu(f(\Om))<\infty$ of Theorem \ref{thm:main theorem}, note that 
as a continuous mapping $f$ is bounded in every $\Om'\Subset\Om$
(i.e. $\overline{\Om'}$ is a compact subset of $\Om$).
Thus we have $\nu(f(\Om'))<\infty$ and also $f\in L^1(\Om';Y)$,
and so we get in fact $f\in N^{1,1}(\Om';Y)$.
\end{remark}

\section{Examples and applications}\label{sec:examples}

There are many examples of spaces equipped with measures that satisfy the asymptotic conditions
of Theorem \ref{thm:main theorem},
even though these measures are not globally or even locally doubling or Ahlfors regular.
To begin with, we consider the following simple example involving weights;
by a weight we mean a nonnegative locally integrable function.

\begin{example}\label{ex:weights}
	Suppose $(X_0,d,\mu_0)$ and $(Y,d_Y,\nu)$ are complete Ahlfors $Q$-regular spaces, with $Q>1$.
	Define $X$ to be the same metric space as $X_0$ but equipped with a weighted measure
	$d\mu:=w\,d\mu_0$.
	Suppose $\Om\subset X$ is open and bounded and
	that there exist sets $E_1,E_2\subset \Om$ such that $E_1\cap E_2=\emptyset$,
	$E_1$ is at most countable and $w|_{X\setminus E_1}$ finite and continuous,
	and $E_2$ has $\sigma$-finite $\mathcal H^{Q-1}$-measure and $w>0$ in $X\setminus E_2$.
	Using the continuity of $w|_{X\setminus E_1}$, it is straightforward to check that
	$E_2$ then has also $\sigma$-finite $\widehat{\mathcal H}$-measure.
	Take $E:=E_1\cup E_2$.
	Suppose
	$f\colon \Om\to Y$ is injective and continuous with $\nu(f(\Om))<\infty$.
	For every $x\in\Om\setminus E$, by the continuity of  $w|_{X\setminus E_1}$ we have
	\[
	\limsup_{r\to 0}\frac{\mu(B(x,2r))}{\mu(B(x,r))}
	=\limsup_{r\to 0}\frac{w(x)\mu_0(B(x,2r))}{w(x)\mu_0(B(x,r))}\le 2^Q C_A^2,
	\]
	where $C_A\ge 1$ is the Ahlfors regularity constant of $X_0$ and $Y_0$, and so
	$\mu$ is asymptotically doubling in $\Om\setminus E$.
	Choosing $R(x):=2C_A^2w(x)$,
	Condition (3) of Theorem \ref{thm:main theorem intro} becomes
	\[
	(w(\cdot)h_f^{\vee}(\cdot)^{Q})^{1/(Q-1)}\in L^1(\Om,\mu),\quad\textrm{or equivalently}\quad 
	(w(\cdot)h_f^{\vee}(\cdot))^{Q/(Q-1)}\in L^1(\Om,\mu_0).
	\]
	If this is satisfied, and also $h_f<\infty$ in $\Om\setminus E$, then $f\in D^1(\Om;Y)$
	by Theorem \ref{thm:main theorem intro}.
\end{example}

The following type of spaces achieved by ``glueing'' spaces of possibly different dimensions
are commonly considered in analysis on metric spaces.
In such spaces the measure is generally not Ahlfors regular nor even doubling in most cases,
but we can now include these spaces in the theory.

\begin{example}\label{ex:bowtie}
	For each $k=1,\ldots,m$, let $n_k\ge 2$ and
	\[
	X_k:=\{x=(x_1,\ldots,x_{n_k})\in\R^{n_k}\colon x_j\ge 0\textrm{ for all }j=1,\ldots,n_k\}.
	\]
	Define the space $X=Y$ by ``glueing'' the above spaces together at the origin $\{0\}$;
	that is, the metric is the Euclidean metric in each $X_k$, and
	$d(x,y):=|x|+|y|$ for $x\in X_k$ and $y\in X_l$, $k\neq l$.
Equip each $X_k$ with the
weighted Lebesgue measure $d\mu:=w\,d\mathcal L^{n_k}$, where
$w\colon X\setminus \{0\}\to (0,\infty)$ is any continuous function such that $\mu$ becomes locally finite
(finite in a neighborhood of the origin, that is).
Suppose $\Om\subset X$ is open and bounded and
$f\colon \Om\to  Y$ is injective and continuous with $\nu(f(\Om))<\infty$, and
$f(\Om\cap X_k)\subset X_k$
for each $k=1,\ldots,m$.
Let $E$ be the union of the origin and sets of $\sigma$-finite $\mathcal H^{n_k-1}$-measure
in $X_k\setminus \{0\}$, $k=1,\ldots,m$.
Now Condition (3) of Theorem \ref{thm:main theorem intro} is
\begin{equation}\label{eq:condition 3 example}
\sum_{k=1}^{m}\int_{X_k\cap \Om}\left(\frac{w(x)}{w(f(x))}
 h_f^{\vee}(x)^{n_k}\right)^{1/(n_k-1)}w(x)\,d\mathcal L^{n_k}(x)<\infty.
\end{equation}
If this is satisfied, and also $h_f<\infty$ in $\Om\setminus E$, then $f\in D^{1}(\Om;Y)$
by Theorem \ref{thm:main theorem intro}.
For example, take  $m=2$ and $n_1=n_2=2$, to obtain the ``bowtie''.
Typically, one considers $w(x):=|x|^{\alpha}$, $\alpha\in\R$.
Now \eqref{eq:condition 3 example} is simply
\[
\int_{\Om\cap X_1} \frac{|x|^{2\alpha}}{|f(x)|^{\alpha}}h^{\vee}_f(x)^2\,d\mathcal L^2(x)
+\int_{\Om\cap X_2} \frac{|x|^{2\alpha}}{|f(x)|^{\alpha}}h^{\vee}_f(x)^2\,d\mathcal L^2(x)
<\infty.
\]
Note that if $-2<\alpha<-1$, we have $\widehat{\mathcal{H}}(\{0\})=\infty$,
but $\{0\}$ can always be included in $E$ because it is only one point,
recall Proposition \ref{prop:exceptional set E}.
\end{example}

Now consider the following:
we equip Ahlfors $Q$-regular metric spaces $(X_0,d,\mu_0)$ and $(Y_0,d_Y,\nu_0)$ with weights $w$ and $w_Y$.
Note that the weights have Lebesgue points almost everywhere, see e.g. Heinonen \cite[Theorem 1.8]{Hei}.
For our purposes, it is natural to consider the pointwise representatives
\begin{equation}\label{eq:w representative}
w(x)=\limsup_{r\to 0}\frac{1}{\mu_0(B(x,r))}\int_{B(x,r)}w\,d\mu_0,\quad x\in X,
\end{equation}
and
\begin{equation}\label{eq:wY representative}
w_Y(y)=\liminf_{r\to 0}\frac{1}{\nu_0(B(y,r))}\int_{B(y,r)}w_Y\,d\nu_0,\quad y\in Y.
\end{equation}

It is straightforward to check that $w,w_Y$ are then Borel functions.

\begin{proof}[Proof of Corollary \ref{cor:weights}]
In order to apply Theorem \ref{thm:main theorem}, we essentially only need to
take care of the problem that $\mu$ might not be asymptotically doubling at every point $x\in\Om\setminus E$.
Denote this ``bad'' set by $A$.
By the fact that
$\mu_0$ is Ahlfors regular and $w$ has Lebesgue points $\mu_0$-a.e., we have $\mu_0(A)=0$.

From the definition of $\widehat{\mathcal H}$ it is straightforward to check
that $\mu(E)=0=\widetilde{\mu}(E)$, so in particular $E$ is $\widetilde{\mu}$-measurable. 
We have $A=\bigcup_{j=1}^{\infty}A_j$ where $w\ge 1/j$ in $U_j$ for some
open $U_j\supset A_j$.
Consider also the sets $D_j$ where
\begin{equation}\label{eq:Dj}
w(x)=\limsup_{r\to 0}\frac{1}{\mu_0(B(x,r))}\int_{B(x,r)}w\,d\mu_0\le  j\quad\textrm{for all }x\in D_j.
\end{equation}
Note that for all $x\in D_{j}$, necessarily
\begin{equation}\label{eq:density estimate}
\liminf_{r\to 0}\frac{\mu_0(D_{2j}\cap B(x,r))}{\mu_0(B(x,r))}\ge \frac{1}{2}.
\end{equation}
Recursively define $w_0:=w$ and
\[
w_{j+1}:=
\begin{cases}
2j & \textrm{in }U_j\cap D_{2j}\\
w_j & \textrm{in }\Om\setminus (U_j\cap D_{2j}).
\end{cases}
\]
Choosing each $U_j$ small enough, we can ensure that
\[
\int_{\Om}w_{j+1}\,d\mu_0\le \int_{\Om}w_{j}\,d\mu_0+2^{-j}
\]
as well as
\[
\int_{\Om}\left(\frac{[w_{j+1}(\cdot)h_f^{\vee}(\cdot)]^{Q}}{w_Y(f(\cdot))}\right)^{1/(Q-1)}\,d\mu_0
\le \int_{\Om}\left(\frac{[w_{j}(\cdot)h_f^{\vee}(\cdot)]^{Q}}{w_Y(f(\cdot))}\right)^{1/(Q-1)}\,d\mu_0
+2^{-j}.
\]
Noting that $w_j$ is an increasing sequence, at the limit we can define a weight function
$\widetilde{w}:=\lim_{j\to\infty}w_j$ a.e., and then we can take the pointwise representative
\[
\widetilde{w}(x)=\limsup_{r\to 0}\frac{1}{\mu_0(B(x,r))}\int_{B(x,r)}\widetilde{w}\,d\mu_0,\quad x\in X.
\]
By monotone convergence, we still have $\widetilde{w}\in L^1(\Om,\mu_0)$ and
\[
\left(\frac{[\widetilde{w}(\cdot)h_f^{\vee}(\cdot)]^{Q}}{w_Y(f(\cdot))}\right)^{1/(Q-1)}\in L^1(\Om,\mu_0)
\quad\textrm{and so}\quad
\left(\frac{\widetilde{w}(\cdot)h_f^{\vee}(\cdot)^{Q}}{w_Y(f(\cdot))}\right)^{1/(Q-1)}\in L^1(\Om,\widetilde{\mu}).
\]
Let $d\widetilde{\mu}:=\widetilde{w}\,d\mu_0$; then $\mu\le \widetilde{\mu}\ll\mu$.
Now for every $x\in U_j\cap D_j$, by \eqref{eq:density estimate} we have
\[
\liminf_{r\to 0}\frac{\widetilde{\mu}(B(x,r))}{\mu_0(B(x,r))}
\ge \liminf_{r\to 0}\frac{2j\mu_0(D_{2j}\cap B(x,r))}{\mu_0(B(x,r))}
\ge j.
\]
Now for every $x\in A_j\cap D_j$ we have 
\[
\limsup_{r\to 0}\frac{\mu(B(x,20r))}{\widetilde{\mu}(B(x,r))}
\le \limsup_{r\to 0}\frac{1}{j\mu_0(B(x,r))}\int_{B(x,20r)}w\,d\mu_0
\le 20^Q  C_A^2 
\]
by \eqref{eq:Dj}, where $C_A\ge 1$ is the Ahlfors regularity constant of $X_0$ and $Y_0$.
Since $A=\bigcup_{j=1}^{\infty}A_j\cap D_j$, we have the above for all $x\in A$ and then in fact for
all $x\in \Om\setminus E$.
By making $w_Y$ larger if necessary, we can assume that $w_Y>0$ in $f(\Om)$,
such that still $\nu(f(\Om))<\infty$.
Now we can choose $R(x)$ in Theorem \ref{thm:main theorem}
to be $2C_A^2\widetilde{w}(x)/w_Y(f(x))$, for all $x\in\Om$.
Thus Theorem \ref{thm:main theorem}
combined with Proposition \ref{prop:exceptional set E} gives the conclusion.
\end{proof}

As we have already seen, our results enable equipping the space $X$ with various weights.
However, Corollary \ref{cor:weights} is especially flexible when applied to weights in the
space $Y$. In this case, we can obtain results that seem to be new already
in unweighted Euclidean spaces, as in the following example.

\begin{example}\label{ex:plane}
	Consider the  square $\Om:=(-1,1)\times (-1,1)$ on the unweighted plane $X=Y=\R^2$.
	Let $1<b<\infty$ and consider the homeomorphism $f\colon \Om\to \Om$
\[
f(x_1,x_2):=
\begin{cases}
(x_1,x_2^{b}), & x_2\ge 0,\\
(x_1,-|x_2|^{b}), & x_2\le 0.
\end{cases}
\]
Essentially, $f$ maps squares centered at the origin to rectangles that become
thinner and thinner near the origin.

By symmetry, it will be enough to study the behavior of $f$ in the unit square $S:=(0,1)\times (0,1)$.
There we have
\[
Df(x_1,x_2)=
\left[
\begin{matrix}
1 & 0 \\
0 & b x_2^{b-1}
\end{matrix}
\right]
\]
and so
\[
|Df|=\sqrt{(1+(b x_2^{b-1})^2}\le 1+b x_2^{b-1},
\]
so that $Df\in L^1(S)$ and then in fact $Df\in L^1(\Om)$ and $f\in N^{1,1}(\Om;\Om)$.

Clearly $f$ maps a small square centered at $(x_1,x_2)\in S$ with side length $\eps$
to a rectangle centered at $f(x_1,x_2)$ and with side lengths
$\eps$ and $b x_2^{b-1}\eps+o(\eps)$.
This means that
\[
h_f(x_1,x_2)= b^{-1}x_2^{1-b}\quad\textrm{in }S_1:=
\{(x_1,x_2) \in S\colon x_2\le b^{1/(1-b)}\}
\]
and
\[
h_f(x_1,x_2)= bx_2^{b-1}\quad\textrm{in }S_2:=
\{(x_1,x_2) \in S\colon x_2\ge b^{1/(1-b)}\}.
\]
Of these two sets, $S_1$ is the relevant one for us, since it contains a neighborhood of
the $x_1$-axis in $S$.

Obviously the plane is an  Ahlfors $Q$-regular space with $Q=2$, and so
\[
h_f^{Q/(Q-1)}(x_1,x_2)=h_f^2(x_1,x_2)= b^{-2}x_2^{2-2b}
\]
for $(x_1,x_2)\in S_1$, which is not in
$ L^1(S_1)$ when  $b\ge 3/2$.
Thus $h_f$ is not in $L^{Q/(Q-1)}(\Om)$, and so
Williams \cite[Corollary 1.3]{Wi} does not give $f\in N^{1,1}(\Om;\Om)$.

In this classical setting a sharper result exists: in Koskela--Rogovin
\cite[Corollary 1.3]{KoRo} the condition is
$H_f\in L^{1}(\Om)$; the definition of $H_f$ was given in \eqref{eq:Hf}
but note that in this example we have $h_f=H_f$.
However, now $h_f\notin L^1(S_1)$ when $b\ge 2$.

On the other hand, we can equip $Y$ with the weight $w_Y(y_1,y_2)=|y_2|^{u}$, $-1<u<0$.
Now for $(x_1,x_2)\in \Om$,
\[
w_Y(f(x))=||x_2|^{b}|^{u}=|x_2|^{bu},
\]
and so in $S_1$,
\begin{equation}\label{eq:weight quantity}
\left(\frac{h_f(x_1,x_2)^{Q}}{w_Y(f(x_1,x_2))}\right)^{1/(Q-1)}
=b^{-2}x_2^{2-2b}x_2^{-bu},
\end{equation}
which is in $L^1(S_1)$ if $b<3/(2+u)$.
Varying $-1<u<0$, we conclude that $b<3$ is required.
Since $S_1$ contains a neighborhood of the $x_1$-axis in $S$, clearly
the left-hand side of \eqref{eq:weight quantity} is then in $L^1(S)$, and then by symmetry in $L^1(\Om)$.
Moreover, $h_f<\infty$ in $\Om\setminus E$ with $E$ consisting of the $x_1$-axis intersected with $\Om$, so that
$E$ obviously has finite $\widehat{\mathcal H}$-measure ($\widehat{\mathcal H}$ being comparable to the
$1$-dimensional Hausdorff measure).
Thus Corollary \ref{cor:weights} implies that $f\in N^{1,1}(\Om;\Om)$.

Hence, we were able to detect that $f\in N^{1,1}(\Om;\Om)$
also in the case $b\in [2,3)$, which was not indicated by
any of the previous results, to the best of our knowledge.
This improvement seems quite unexpected since the Euclidean theory appeared rather complete:
Kallunki--Martio \cite{KaMa} give an example of a mapping $f\colon \R^2\to \R^2$ for which
$H_f\in L_{\loc}^s(\R^2)$ for all $0<s<1$, but $f\notin N_{\loc}^{1,1}(\R^2;\R^2)$.
Thus the condition $H_f\in L_{\loc}^1(\R^2)$ is known to be sharp in the range of integrability exponents.
\end{example}

\begin{remark}
In quasiconformal theory, one is often interested in whether $f\in N^{1,Q}(\Om;Y)$.
Equipping $Y$ with weights is likely to prove useful for this question also,
but establishing $N^{1,Q}$-regularity requires different methods from ours
and probably more assumptions on the space $X$, so in the current paper we only consider the case $p=1$.
\end{remark}

In many previous works including Balogh--Koskela--Rogovin \cite{BKR},
the stronger assumption $h_f\in L^{\infty}(X)$ is made, instead of only $h_f\in L_{\loc}^{Q/(Q-1)}(X)$.
In such a setting, we can obtain for example
the following result.

\begin{corollary}\label{cor:Linfinity}
	Let $(X_0,d,\mu_0)$ and $(Y,d_Y,\nu)$ be Ahlfors $Q$-regular spaces, with $Q>1$.
	Let $X$ be the same metric space as $X_0$, but equipped with the weighted measure
	$d\mu:=w\,d\mu_0$, with the weight $w$ represented by \eqref{eq:w representative}.
	Let $\Om\subset X$ be open and bounded and
	let $f\colon \Om\to Y$ be injective and continuous with $\nu(f(\Om))<\infty$.
	Suppose $E\subset \Om$ is the disjoint union of a countable set $E_1$ and a set $E_2$ with
	$\sigma$-finite $\mathcal H^{Q-1}$-measure, and that
	$w<\infty$ in $\Om\setminus E_1$
	and $h_f<\infty$ in $\Om\setminus E$, and $w\in L^{1/(Q-1)}(\Om,\mu_0)$
	and $h_f\in L^{\infty}(\Om)$.
	Then $f\in D^1(\Om;Y)$.
\end{corollary}

Theorem 1.1 of \cite{BKR} is essentially the above with $w\equiv 1$.
We instead show that very general weights $w$ are allowed: the restrictions are that
$w$ can only have countable many singularities (where $w=\infty$)
and it needs to be $1/(Q-1)$-integrable.

\begin{proof}
By replacing $w$ with $w+1$, we can assume that $w$ is bounded away from zero;
note that if $f\in D^1(\Om;Y)$ with respect to a larger weight, then also
$f\in D^1(\Om;Y)$ with respect to the original weight.
We need to show that $E_2$ has $\sigma$-finite $\widehat{\mathcal H}$-measure.
It is enough to do this for a subset of $E_2$ with finite $\mathcal H^{Q-1}$-measure
and so we can assume that $\mathcal H^{Q-1}(E_2)<\infty$.
Let
\[
A_{j}:=\Big\{x\in E_2\colon \sup_{0<r\le 1/j} \frac{1}{\mu_0(B(x,r))}\int_{B(x,r)}w\,d\mu_0< j\Big\},
\]
and note that $E_2=\bigcup_{j=1}^{\infty}A_{j}$.
It is then enough to show that a fixed $A_{j}$ has finite $\widehat{\mathcal H}$-measure.
Let $0<\delta<1/(2j)$. We find a covering $\{B_l=B(x_l,r_l)\}_{l=1}^{\infty}$ of $A_{j}$
with $r_l\le \delta/2$ and
\[
\sum_{l=1}^{\infty}r_l^{Q-1}
\le \mathcal H^{Q-1}(A_{j})+1
<\infty.
\]
For every $B_l$, we can assume there is a point $y_l\in B_l\cap A_{j}$. Now
the balls $B(y_l,2r_l)$ cover $A_{j}$ and so by Definition \ref{def:centered codimension one measure} of
$\widehat{\mathcal H}$, we get
\begin{align*}
\widehat{\mathcal H}_{\delta}(A_{j})
\le \sum_{l=1}^{\infty}\frac{\mu(B(y_l,4r_l))}{2r_l}
\le j\sum_{l=1}^{\infty}\frac{\mu_0(B(y_l,4r_l))}{2r_l}
&\le jC_A 4^{Q}\sum_{l=1}^{\infty}r_l^{Q-1}\\
&\le jC_A 4^{Q}(\mathcal H^{Q-1}(A_{j})+1),
\end{align*}
where $C_A\ge 1$ is the Ahlfors regularity constant.
Letting $\delta\to 0$, we get $\widehat{\mathcal H}(A_{j})<\infty$ and so by the above,
we conclude that $E_2$ has $\sigma$-finite $\widehat{\mathcal H}$-measure.
Now Corollary \ref{cor:weights} implies the result.
\end{proof}


\begin{thebibliography}{ACMM}

\bibitem{AFP}L. Ambrosio, N. Fusco, and D. Pallara,
\textit{Functions of bounded variation and free discontinuity problems.}
Oxford Mathematical Monographs. The Clarendon Press, Oxford University Press, New York, 2000. xviii+434 pp.

\bibitem{BaKo}Z. Balogh and P. Koskela,
\textit{Quasiconformality, quasisymmetry, and removability in Loewner spaces.}
With an appendix by Jussi V\"ais\"al\"a.
Duke Math. J. 101 (2000), no. 3, 554--577.

\bibitem{BKR}Z. Balogh, P. Koskela, and S. Rogovin,
\textit{Absolute continuity of quasiconformal mappings on curves},
Geom. Funct. Anal. 17 (2007), no. 3, 645--664.

\bibitem{BA}A. Beurling and L. Ahlfors,
\textit{The boundary correspondence under quasiconformal mappings},
Acta Math. 96 (1956), 125--142.

\bibitem{BB}A. Bj\"orn and J. Bj\"orn,
\textit{Nonlinear potential theory on metric spaces},
EMS Tracts in Mathematics, 17. European Mathematical Society (EMS), Z\"urich, 2011. xii+403 pp.

\bibitem{DEKS}E. Durand-Cartagena, S. Eriksson-Bique, R. Korte, and N. Shanmugalingam,
\textit{Equivalence of two BV classes of functions in metric spaces, and existence of a Semmes family of curves under a 1-Poincar\'e inequality},
Adv. Calc. Var. 14 (2021), no. 2, 231--245.

\bibitem{Ge62}F. W. Gehring,
\textit{Rings and quasiconformal mappings in space},
Trans. Amer. Math. Soc. 103 (1962), 353--393.

\bibitem{Ge}F. W. Gehring, 
\textit{The $L^p$-integrability of the partial derivatives of a quasiconformal mapping},
Acta Math. 130 (1973), 265-277.

\bibitem{Haj}P.~Haj{\l}asz,
\textit{Sobolev spaces on metric-measure spaces. Heat kernels and analysis on manifolds, graphs, and metric spaces},
(Paris, 2002), 173--218,
\newblock{Contemp. Math.}, 338, Amer. Math. Soc., Providence, RI, 2003.

\bibitem{Hei}J. Heinonen,
\textit{Lectures on analysis on metric spaces},
Universitext. Springer-Verlag, New York, 2001. x+140 pp.

\bibitem{HK1}J. Heinonen and P. Koskela,
\textit{Definitions of quasiconformality},
Invent. Math. 120 (1995), no. 1, 61--79.

\bibitem{HK2}J. Heinonen and P. Koskela,
\textit{Quasiconformal maps in metric spaces with controlled geometry},
Acta Math. 181 (1998), no. 1, 1--61.

\bibitem{HKST}J. Heinonen, P. Koskela, N. Shanmugalingam, and J. Tyson,
\textit{Sobolev classes of Banach space-valued functions and quasiconformal mappings},
J. Anal. Math. 85 (2001), 87--139.

\bibitem{HKSTbook}J. Heinonen, P. Koskela, N. Shanmugalingam, and J. Tyson,
\textit{Sobolev spaces on metric measure spaces.},
An approach based on upper gradients. New Mathematical Monographs, 27. Cambridge University Press, Cambridge, 2015. xii+434 pp.

\bibitem{KaKo}S. Kallunki and P. Koskela,
\textit{Exceptional sets for the definition of quasiconformality},
Amer. J. Math. 122 (2000), no. 4, 735--743.

\bibitem{KaMa}S. Kallunki and O. Martio,
\textit{ACL homeomorphisms and linear dilatation},
Proc. Amer. Math. Soc. 130 (2002), no. 4, 1073--1078.

\bibitem{Kech}A. Kechris,
\textit{Classical descriptive set theory},
Graduate Texts in Mathematics, 156. Springer-Verlag, New York, 1995. xviii+402 pp.

\bibitem{KoRo}P. Koskela and S. Rogovin,
\textit{Linear dilatation and absolute continuity},
Ann. Acad. Sci. Fenn. Math. 30 (2005), no. 2, 385--392.

\bibitem{MaMo}G. A. Margulis and G. D.  Mostow,
\textit{The differential of a quasi-conformal mapping of a Carnot-Carathéodory space},
Geom. Funct. Anal. 5 (1995), no. 2, 402--433.

\bibitem{Mar}O. Martio,
\textit{The space of functions of bounded variation on curves in metric measure spaces},
Conform. Geom. Dyn. 20 (2016), 81--96.

\bibitem{Mir}
M. Miranda, Jr.,
\textit{Functions of bounded variation on "good'' metric spaces},
J. Math. Pures Appl. (9) 82 (2003), no. 8, 975--1004.

\bibitem{Wi}
M. Williams,
\textit{Dilatation, pointwise Lipschitz constants, and condition N on curves},
Michigan Math. J. 63 (2014), no. 4, 687--700.

\end{thebibliography}
\end{document}